\documentclass[12pt, a4paper]{amsart}
\usepackage{amsmath}
\usepackage{amsxtra}
\usepackage{amscd}
\usepackage{amsthm}
\usepackage{amsfonts}
\usepackage{amssymb}
\usepackage {pstricks}
\usepackage{pstricks,pst-node}

\newtheorem{theorem}{Theorem}[section]
\newtheorem*{theorem*}{Theorem}
\newtheorem{lemma}[theorem]{Lemma}
\newtheorem{proposition}[theorem]{Proposition}
\newtheorem*{proposition*}{Proposition}
\newtheorem{corollary}[theorem]{Corollary}
\newtheorem{definition}[theorem]{Definition}

\theoremstyle{definition}
\newtheorem{remark}[theorem]{Remark}
\newtheorem{example}[theorem]{Example}

\numberwithin{equation}{section}

\def\1{\hbox{1\kern-.35em\hbox{1}}}

\newcommand{\Z}{{\mathbb Z}}

\newcommand{\C}{{\mathbb C}}

\newcommand{\fg}{{\mathfrak g}}

\newcommand{\fl}{{\mathfrak l}}

\newcommand{\fgl}{{\mathfrak {gl}}}

\newcommand{\id}{{\rm id}}
\newcommand{\End}{{\rm End}}
\newcommand{\Hom}{{\rm Hom}}
\newcommand{\Tor}{{\rm Tor}}
\newcommand{\Ext}{{\rm Ext}}

\newcommand{\U}{{\rm U}}

\newcommand{\Ul}{{{\rm U}_q(\fl)}}
\newcommand{\Uq}{{{\rm U}_q(\fg)}}

\newcommand{\cA}{{\mathcal A}}

\newcommand{\cE}{{\mathcal E}}
\newcommand{\cF}{{\mathcal F}}
\newcommand{\cS}{{\mathcal S}}
\newcommand{\cP}{{\mathcal P}}

\newcommand{\cO}{{\mathcal O}}
\newcommand{\cM}{{\mathcal M}}
\newcommand{\cN}{{\mathcal N}}

\newcommand{\QP}{{\mathcal{QP}}}

\newcommand{\Ulmod}{\mbox{$\Ul$-\rm{\bf mod}}}
\newcommand{\Umod}{\mbox{$\U$-\rm{\bf mod}}}

\newcommand{\lr}{{\longrightarrow}}
\newcommand{\Mgr}{{\mathcal M}_{gr}}

\begin{document}

\title[Equivariant K-theory of quantum group actions]
{Quantum group actions on rings and \\ equivariant
K-theory}


\author{G. I. Lehrer and R. B. Zhang}
\address{School of Mathematics and Statistics,
University of Sydney, NSW 2006, Australia.}
\email{gusl@maths.usyd.edu.au, rzhang@maths.usyd.edu.au}

\begin{abstract}
Let $\Uq$ be a quantum group.
Regarding a (noncommutative) space with $\Uq$-symmetry as a $\Uq$-module 
algebra $A$, we may think of equivariant vector bundles on $A$ as projective
$A$-modules with compatible $\Uq$-action. We construct an equivariant K-theory
of such quantum vector bundles using Quillen's exact categories, and provide means for 
its compution. The equivariant K-groups of quantum homogeneous spaces and quantum
symmetric algebras of classical type are computed.
\end{abstract}
\maketitle

\tableofcontents

\section{Introduction} 

In recent years there has been much work
exploring various types of noncommutative geometries with possible
applications in physics. The best developed theory is Connes'
noncommutative differential geometry \cite{Co} formulated within the
framework of $C^*$-algebras, which incorporates K-theory and cyclic
cohomology and has yielded new index theorems. Various aspects of
noncommutative algebraic geometry have also been developed (see,
e.g., \cite{AZ,SvdB}).

Noncommutative generalisations of classical geometries are based on
the strategy of regarding a space as defined by an algebra of functions,
which is commutative in the classical case.
In noncommutative geometry \cite{Co} one
replaces this commutative algebra by noncommutative algebra; in
analogy with the classical case, vector
bundles are regarded as finitely generated projective modules over this
algebra. One may then
investigate problems with geometric origins by means of these algebraic
structures. This permits cross fertilisation of algebraic and
geometric ideas, and is expected to lead to mathematical advances in both
areas. An important motivation from physics for studying noncommutative
geometry is the notion that spacetime at the Planck scale
becomes noncommutative \cite{DFR}; thus noncommutative geometry may
be a necessary ingredient for a consistent theory of quantum
gravity.

Much of classical algebraic  and differential geometry concerns algebraic
varieties and manifolds with algebraic or Lie group actions. Correspondingly,
in noncommutative geometry one studies noncommutative
algebras with Hopf algebra actions.  Natural examples of
such noncommutative geometries are quantum analogues \cite{GZ, Z,
ZZ, LZZ} of homogeneous spaces and homogeneous vector bundles \cite{Bo},
which have proved useful for formulating the quantum group version
\cite{APW} of the Bott-Borel-Weil theorem \cite{Bo} into a
noncommutative geometric setting \cite{GZ, Z}. Some quantum
homogeneous spaces such as quantum spheres have been particular 
objects of attention  \cite{DLPS, Maj, LZZ} because of potential physical
applications. We note that interesting examples of Hopf algebras
acting on noncommutative algebras are noncocommutative,  and
thus do not correspond to groups.

In this paper we study noncommutative geometries with quantum group
symmetries. In particular, we shall study an equivariant
algebraic K-theory of 
such spaces which is a generalisation of
the equivariant algebraic K-theory of reductive group actions
investigated by Bass and Haboush in \cite{BH1, BH2}.
The equivariant topological K-theory of Lie group actions and
algebraic group scheme actions have been developed in the celebrated
papers \cite{Seg, Dp, Th}. Following the work of Bass and Haboush,
several authors have addressed geometric themes in the context of
the equivariant K-theory of algebraic vector bundles \cite{Be,MM,MMP}.
A recent treatment of the K-theory of
compact Lie group actions in relation to representation theory may be found in
\cite[\S 12]{Ku}.

The need for an equivariant K-theory in the noncommutative setting
was already clear in the classification of quantum homogeneous
bundles \cite{ZZ}. Very recently, Nest and Voigt have extended the notion
of Poincar´e duality in K-theory to the setting of compact quantum
group actions \cite{Nest} within the framework of $C^*$-algebras.

In our situation, the two crucial notions which are needed are 
those of an `equivariant noncommutative space', which we shall
take to mean a module algebra
$A$ over a Hopf algebra \cite{M}, and an `equivariant vector
bundle on $A$' which we shall take to mean an equivariant module
over the module algebra $A$. Specifically, a module algebra is an associative
algebra that is also a module for the Hopf algebra, whose
algebraic structure is preserved by the action; equivariant modules over module
algebras were introduced in \cite{ZZ} in the context of
quantum homogeneous spaces, and we generalise it here to arbitrary module
algebras.

Let $\U=\U_q(\fg)$ be a quantum group defined over the field $\Bbbk=\C(q)$,
and let $A$ be a module algebra over $\U$. We
introduce the category $\cM(A, \U)$ of $\U$-equivariant $A$-modules
which are finitely $A$-generated and locally $\U$-finite. The full
subcategory $\cP(A, \U)$ of $\cM(A, \U)$ consisting of
finitely $A$-generated, locally $\U$-finite, projective equivariant modules
is an exact category in a natural way. Thus Quillen's K-theory
of exact categories applies to $\cP(A, \U)$, giving rise to an
algebraic K-theory of module algebras which is equivariant under the
action of the quantum group $\U$ (Definition \ref{equivKgps}).

Properties of equivariant K-theory are developed for
module algebras with filtrations which are stable under quantum group
actions. Under a regularity assumption for the module algebras, we
establish a relationship between the K-groups of the filtered
algebras and the degree zero subalgebras of the corresponding graded
algebras (see Theorem \ref{filtered} for details).  This may be
regarded as an equivariant analogue of \cite[\S 6, Theorem 7]{Q} on
the higher algebraic K-theory of filtered rings.

We apply Theorem \ref{filtered} to compute the equivariant K-groups
for a class of module algebras over quantum groups, which we call
quantum symmetric algebras. The results are summarised in Theorem
\ref{q-symmalg}. In establishing the regularity of the left
Noetherian quantum symmetric algebras, some elements of the theory
of Koszul algebras \cite{PP} are used, which are discussed in
Section \ref{Koszul-complex}.

One of the main motivations of \cite{BH2} is to
prove the results \cite[Theorems 1.1,2.3, Cor. 2.4] {BH2}, which 
would be easy consequences of the Serre conjecture
if one assumed that all reductive group actions on affine space are linearisable.
Thus a natural question which arises from our work is 
whether there is a non-commutative version of the Serre conjecture
for the quantum symmetric algebras we consider. In \cite{ART}, the 
case of the natural representation of $\U_q(\fgl_n)$ is discussed.

We also study the equivariant K-theory of quantum homogeneous spaces 
in detail. Given a quantum homogeneous space of a quantum
group $\Uq$ which corresponds to a reductive quantum subgroup $\Ul$, we
show that the equivariant K-groups are isomorphic to the K-groups of
the exact category whose objects are the $\Ul$-modules (see
Corollary \ref{main-1} for the precise statement).

Properties of the categories $\cM(A, \U)$
and $\cP(A, U)$, analogous to those of their classical analogues
are established, and used in the study of equivariant K-theory.
We prove a splitting lemma (Proposition
\ref{splitting}), which enables us to characterise the finitely $A$ generated,
locally $\U$-finite,
projective $\U$-equivariant $A$-modules (Corollary \ref{projective}). A
similar result in the commutative setting was proved by Bass and
Haboush \cite{BH1, BH2} for reductive algebraic group actions.

The equivariant algebraic K-theory constructed here
generalises, in a completely straightforward manner,
to Hopf algebras whose locally finite modules
are semi-simple, e.g.,
universal enveloping algebras of finite dimensional semi-simple Lie
algebras. In fact, the case of such universal enveloping algebras
essentially covers the Bass-Haboush theory for semi-simple
algebraic group actions when the
module algebras are commutative.
We also point out that the equivariant algebraic K-theory of
a $\U$-module algebra $A$ developed here is
different from the usual algebraic K-theory of the smash product
algebra $R:=A\#\U$, see Remark \ref{not-smash}.

The organisation of the paper is as follows. In Section
\ref{definitions}, we introduce various categories of equivariant modules
for module algebras over quantum groups, and define the equivariant algebraic K-theory
of quantum group actions. In Section \ref{equiv-modules}, the
theory of equivariant modules is developed, and
is used to study quantum group equivariant K-theory.
In Section \ref{filtered-algebras} we develop the equivariant K-theory
of filtered module algebras, and in the remaining two sections
we study concrete examples.  In Section
\ref{quantum-symmalg}, we compute the equivariant K-groups of the
quantum symmetric algebras, and in Section \ref{q-homo-spaces} we
investigate in detail the equivariant K-groups of quantum
homogeneous spaces.

\section{Equivariant K-theory of quantum group actions}\label{definitions}

The purpose of this section is to introduce an equivariant algebraic K-theory
of quantum group actions. This theory generalises,
in a straightforward way to  arbitrary Hopf algebras.

\subsection{Module algebras and equivariant modules}
For any finite dimensional simple complex Lie algebra $\fg$,
denote by $\U:=\U_q(\fg)$ the quantum group defined over the field $\Bbbk=\C(q)$ of
rational functions in $q$;
$\U$ has a standard presentation with generators $\{e_i, f_i, k_i^{\pm 1} \mid
i=1,\dots,r\}$ ($r=\text{rank}(\fg)$), and relations which may be found e.g., in \cite{APW}.
If $\fg$ is a semi-simple Lie algebra, $\U$ will denote the tensor product
of the quantum groups in the above sense, associated with the simple factors.

It is well known that $\U$ has the structure of a
Hopf algebra; denote its co-multiplication by $\Delta$, co-unit
by $\epsilon$ and antipode by $S$. We
shall use Sweedler's notation for co-multiplication: given any
$x\in \U$, write $\Delta(x)=\sum_{(x)} x_{(1)}\otimes x_{(2)}$.
The following relations are among those satisfied by
any Hopf algebra:
\[
\begin{aligned}
\sum_{(x)} \epsilon(x_{(1)}) x_{(2)} &= \sum_{(x)}
x_{(1)}\epsilon(x_{(2)}) = x,\\
\sum_{(x)} S(x_{(1)}) x_{(2)} &= \sum_{(x)} x_{(1)}S(x_{(2)}) =
\epsilon(x).
\end{aligned}
\]
Let $\Delta'$ be the opposite co-multiplication, defined by
$\Delta'(x)=\sum_{(x)} x_{(2)}\otimes x_{(1)}$ for any $x\in \U$.

Denote by $\Umod$ the category of finite
dimensional left $\U$-modules of type-$(1, \dots, 1)$.
Then $\Umod$ is a semi-simple braided tensor category.
A (left) $\U$-module $V$ is
called {\em locally finite} if for any $v\in
V$, the cyclic submodule $\U v$ generated by $v$ is
finite dimensional. We shall make use of the important 
fact that locally finite modules are semi-simple.
We shall say that a locally finite $\U$-module is type-$(1, \dots,
1)$ if all its finite dimensional submodules are type-$(1, \dots, 1)$.

An associative algebra $A$ with identity $1$ is a (left) {\em module
algebra over $\U$} \cite{M} if it is a left $\U$-module, and the
multiplication $A\otimes_\Bbbk A \longrightarrow A$ and unit map
$\Bbbk \longrightarrow A$ are $\U$-module homomorphisms. 
Explicitly, if we write the $\U$-action on $A$ as $\U\otimes_\Bbbk A
\longrightarrow A$, $x\otimes a \mapsto x\cdot a$, for all $a\in A$
and $x\in \U$, then
\[
\begin{aligned}
x\cdot(a b) = \sum_{(x)} (x_{(1)}\cdot a)(x_{(2)}\cdot b),\quad
x\cdot 1 = \epsilon(x) 1.
\end{aligned}
\]

We call a $\U$-module algebra $A$ {\em locally finite}
if it is locally finite as a $\U$-module. If all its
submodules are in $\U$-{\bf mod}, we say that the locally finite
$\U$-module algebra is type-$(1, \dots, 1)$.

An element $a\in A$ is {\em $\U$-invariant} if $x\cdot a
=\epsilon(x) a$ for all $x\in \U$. We denote by  $A^\U$ the
submodule of $\U$-invariants of $A$, that is,
\[A^\U:= \{a\in A \mid x\cdot a =\epsilon(x) a, \quad  \forall x\in \U\}.\]
The fact that $\U$ is a Hopf algebra implies that this
is a subalgebra of $A$. Indeed, for all $a, b\in A^\U$, we have
\[x\cdot(a b) = \sum_{(x)} (x_{(1)}\cdot a)(x_{(2)}\cdot b)
= \sum_{(x)} \epsilon(x_{(1)})\epsilon(x_{(2)}) a b =\epsilon(x) a
b.
\]
Hence $a b \in A^\U$. We shall refer to $A^\U$ as the {\em
subalgebra of $\U$-invariants} of $A$.

Let $M$ be a left $A$-module with structure map $\phi: A\otimes
M\longrightarrow M$, and also a locally finite left $\U$-module with
structure map $\mu: \U\otimes M \longrightarrow M$.
Then $A\otimes_\Bbbk M$ has a natural $\U$-module structure
\begin{eqnarray*}
\mu': \U\otimes \left(A\otimes M\right)&\longrightarrow& A\otimes M,
\\  x\otimes (a\otimes m) &\mapsto& \sum_{(x)} x_{(1)}\cdot a \otimes
x_{(2)} m.
\end{eqnarray*}
The $A$-module and $\U$-module structures of $M$
are said to be compatible if the following diagram commutes
\begin{equation}\label{compatibility}
\begin{array}{c c c c c c}
&\U\otimes (A\otimes M) & \stackrel{\id\otimes
\phi}{\longrightarrow} & \U\otimes M&\\
& \mu'\downarrow&  &\mu\downarrow&\\
&A\otimes M & \stackrel{\phi}{\longrightarrow} & M.&
\end{array}
\end{equation}
In this case, $M$ is called a {\em $\U$-equivariant left $A$-module}, or
$A$-$\U$-module for simplicity.
A morphism between two $A$-$\U$-modules is an $A$-module map which
is at the same time also $\U$-linear. We denote by $\Hom_\text{$A$-$\U$}(M,
N)$ the space of $A$-$\U$-morphisms.

Denote by $A$-$\U$-{\bf mod} the category of locally
$\U$-finite $A$-$\U$-modules (i.e., locally $\U$-finite
$\U$-equivariant left $A$-modules), which as $\U$-modules are of type-$(1, \dots, 1)$.
It is clear that $A$-$\U$-{\bf mod}
is an abelian category.  Let $\cM(A, \U)$  be the full
subcategory of $A$-$\U$-modules consisting of finitely $A$-generated objects,
and denote by $\cP(A, \U)$ the full subcategory
of $\cM(A, \U)$ whose objects are the projective objects in $A$-$\U$-{\bf mod}.

\begin{remark}\label{type1}
\begin{enumerate}
\item In this work, $\U$ shall generally denote $\U=\U_q(\fg)$,
where $\fg$ is a reductive Lie algebra. In perticular, It may happen that the
root lattice has smaller rank than the weight lattice. 
We shall consider locally finite $\U$-modules and locally finite $\U$-module algebras
which are of type-$(1, \dots, 1)$ for the chosen dominant weights. In 
particular, the categories of finite dimensional modules considered are semisimple.
\item The categories just introduced are quantum analogues of those occurring
in \cite[Theorem 2.3]{BH2}.
\end{enumerate}
\end{remark}

We also define a {\em $\U$-equivariant right $A$-module} $M$,
as a left $\U$-module which is also a
right $A$-module, such that the module structures are compatible in the sense that
\[ x(m a) = \sum_{(x)} (x_{(1)}m)(x_{(2)}\cdot a) \]
for all $x\in\U$, $a\in A$ and $m\in M$.
Similarly, we also define a {\em $\U$-equivariant $A$-bimodule} $M$,
as a left $\U$-module which is also an
$A$-bimodule, such that
\[ x(a m b) = \sum_{(x)} (x_{(1)} \cdot a) (x_{(2)} m(x_{(3)})\cdot b) \]
for all $x\in\U$, $a, b\in A$ and $m\in M$.

Let $R$ be a $\U$-equivariant right $A$-module, and let $B$ be a $\U$-equivariant  $A$-bimodule.
For any $\U$-equivariant left $A$-module $M$, $R\otimes_A M$ has the structure of left $\U$-module,
and $B\otimes_A M$ has the structure of $\U$-equivariant left $A$-module, with the module structures define
in the following way. For all $r\in R$, $b\in B$, $a\in A$ and $m\in M$,
\[
\begin{aligned}
&x(r\otimes_A m) = \sum_{(x)} x_{(1)}r\otimes_A x_{(1)}m,  \\
&x(b\otimes_A m) = \sum_{(x)} x_{(1)}b\otimes_A x_{(1)}m, \\
&a(b\otimes_A m) = ab\otimes_A m.
\end{aligned}
\]

The following result is now clear.
\begin{lemma} \label{tensor-functors} Let $A$ be a locally finite $\U$-module algebra.
Let $R$ be a $\U$-equivariant right $A$-module, and let $B$ be a $\U$-equivariant
$A$-bimodule. Assume that both $R$ and $B$ are locally $\U$-finite,
then we have covariant functors
\[
\begin{aligned}
&R\otimes_A - : \text{$A$-$\U$-{\bf mod}} \longrightarrow \text{$\U$-{\bf Mod}$_{l.f.}$},\\
&B \otimes_A - : \text{$A$-$\U$-{\bf mod}} \longrightarrow \text{$A$-$\U$-{\bf mod}},
\end{aligned}
\]
where $\U$-{\bf Mod}$_{l.f.}$ is the category of locally finite $\U$-modules.
Furthermore, if $B$ is also finitely $A$-generated as a left $A$-module,
then we have the covariant functor
\[B \otimes_A - : \cM(A, \U) \longrightarrow  \cM(A, \U). \]
\end{lemma}

In this work, the term `module' will mean left module 
unless otherwise stated.

\subsection{Equivariant K-theory of quantum group actions}\label{definitions-K}
Recall that an exact category $\cP$ is an additive category with a
class ${\mathbf E}$ of short exact sequences which satisfies a
series of axioms, see \cite[p.99]{Q} or \cite[Appendix A]{K}. For
our purposes, we may think of an exact category $\cP$ as a full
(additive) subcategory of an abelian category $\cA$ which is closed
under extensions in $\cA$, that is, for any short exact sequence
$0\longrightarrow M' \longrightarrow M \longrightarrow M''
\longrightarrow 0$ in $\cA$, if $M'$ and $M''$ are in $\cP$, then
$M$ also belongs to $\cP$. Typical examples of exact categories are
\begin{enumerate}
\item any abelian category with exact structure given by all
short exact sequences,  and 
\item the full
subcategory of finitely generated projects (left) modules of the
category of (left) modules over a ring.
\end{enumerate}
For any exact category $\cP$ in which the
isomorphism classes of objects form a set,
one may define the Quillen category $\QP$.
Quillen's algebraic K-groups \cite{Q} of the exact category $\cP$ are defined
to be the homotopy groups of the classifying space $B(\QP)$ of
$\QP$:
\[
K_i(\cP) = \pi_{i+1}(B(\QP)), \quad i=0, 1, \dots.
\]

If $F: \cP_1 \longrightarrow \cP_2$ is an exact functor between
exact categories, it induces a functor ${\mathcal Q}F: \QP_1
\longrightarrow \QP_2$ between the corresponding Quillen
categories. This functor then induces a cellular map $B{\mathcal
Q}F: B(\QP_1) \longrightarrow B(\QP_2)$, which in turn leads
to the homomorphisms
\[
F_*: K_i(\cP_1) \longrightarrow K_i(\cP_2), \quad \text{for
all $i$}.
\]

We now turn to the definition of an equivariant algebraic K-theory
of quantum group actions.
The following fact is immediate from the definition of $\cP(A, \U)$.
\begin{theorem}
Let $A$ be a locally finite module algebra over the quantum group $\U$.
Then the category $\cP(A, \U)$ of finitely $A$-generated,
locally $\U$-finite, projective $\U$-equivariant $A$-modules
is an exact category.
\end{theorem}

The Quillen K-groups $K_i(\cP(A, \U))$ are therefore defined
for $\cP(A, \U)$, and the following definition makes sense.

\begin{definition}\label{equivKgps}
Let $A$ be a locally finite module algebra over the quantum group $\U$.
The $\U$-equivariant algebraic K-groups of $A$ are defined by
\[
K_i^\U(A): = K_i(\cP(A, \U)), \quad i=0, 1, \dots.
\]
\end{definition}

It follows from standard facts  \cite[Theorem 1, p.102]{Q} that the
fundamental group of $B(\QP(A, \U))$ is isomorphic to the
Grothendieck group of $\cP(A, \U)$. Hence the
$\U$-equivariant K-group $K_0^\U(A)$ is isomorphic to the
Grothendieck group of $\cP(A, \U)$.

\begin{remark}
The $\U$-equivariant algebraic K-groups $K_i^\U(A)$ of $A$ are
a generalisation to the quantum group setting of the
equivariant algebraic K-groups of reductive algebraic group actions
studied in \cite{BH1, BH2}.
\end{remark}

\begin{remark}\label{not-smash}
One may also consider the usual algebraic K-theory of the smash
product $A\#U$ (see Appendix \ref{smash} for the definition of a smash product).
This, however, is completely different from the equivariant K-theory of the $\U$-module algebra
$A$ introduced here. See Appendix \ref{smash} for more details.
\end{remark}

\section{Categories of equivariant modules}\label{equiv-modules}

To study the equivariant K-groups introduced in the last section, we
require some properties of various categories of equivariant modules.
Fix a quantum group $\U$ and a locally finite
module algebra $A$ over $\U$. As we have already declared in Remark \ref{type1},
all locally finite $\U$-modules and locally finite $\U$-module algebras considered
are assumed to be type-$(1, \dots, 1)$.

\begin{lemma} \label{HomA}
Let $M$ and $N$ be $A$-$\U$-modules.  Then there is a natural
$\U$-action on $\Hom_A(M, N)$ defined for any $x\in\U$ and
$f\in\Hom_A(M, N)$ by
\begin{eqnarray}\label{U-Hom}
(x f)(m)=\sum_{(x)}x_{(2)} f(S^{-1}(x_{(1)}) m), \quad \forall m\in
M.
\end{eqnarray}
\end{lemma}
\begin{proof}
We first show that 
\eqref{U-Hom} defines a $\U$-module structure on $\Hom_\Bbbk(M, N)$.
For any $x, y \in \U$ and
$f\in\Hom_\Bbbk(M, N)$, we have
\[
\begin{aligned}
(y(x f))(m)&=\sum_{(y)} y_{(2)} (x f)(S^{-1}(y_{(1)}) m)\\
           &=\sum_{(y), (x)} y_{(2)} x_{(2)} f(S^{-1}(x_{(1)}) S^{-1}(y_{(1)}) m)
\end{aligned}
\]
for all $m\in M$. Using the facts that  for all $x,y\in\U$, $S(y x)
= S(x) S(y)$ and $\Delta(y x)=\sum_{(x), (y)} y_{(2)} x_{(2)}\otimes
y_{(1)} x_{(1)}$, we can cast the far right hand side into
\[
\sum_{(yx)} (y x)_{(2)} f(S^{-1}((y x)_{(1)}) m)=((yx) f)(m)
\]
Thus $\Hom_\Bbbk(M, N)$ is a $\U$-module.

Next we show that $\Hom_A(M, N)$ is a $\U$-submodule of $\Hom_\Bbbk(M, N)$. Let
$f\in\Hom_A(M, N)$, $a\in A$ and $x\in\U$. Then for all $m\in M$, we
have
\[
\begin{aligned}
(x f)(a m) &=\sum_{(x)} x_{(2)} f(S^{-1}(x_{(1)}) (a m))\\
           &=\sum_{(x)} x_{(3)} f((S^{-1}(x_{(2)})\cdot a) S^{-1}(x_{(1)}) m)\\
           &= \sum_{(x)} x_{(3)}((S^{-1}(x_{(2)})\cdot a) f( S^{-1}(x_{(1)}) m)),
\end{aligned}
\]
where the last step used the fact that $f$ is $A$-linear. The far
right hand side can be rewritten as $\sum_{(x)}
(x_{(3)}\cdot(S^{-1}(x_{(2)})\cdot a)) x_{(4)} f( S^{-1}(x_{(1)})
m)$. By using the defining property of the antipode, we can further
simplify it to obtain
\[
\begin{aligned}
&\sum_{(x)} \epsilon(x_{(2)}) a \left(x_{(3)}f( S^{-1}(x_{(1)}) m)\right)\\
&=\sum_{(x)} a \left(x_{(2)}f( S^{-1}(x_{(1)}) m)\right)\\
&= a (x f)(m).
\end{aligned}
\]
Hence $x f\in\Hom_A(M, N)$, as required.
\end{proof}

For any $M, N$ in
$\cM(A, \U)$, the $\U$-action on $\Hom_A(M, N)$ defined in Lemma
\ref{HomA} is semi-simple, and $\Hom_A(M, N)$ belongs to $\cM(A,
\U)$. To see this, we note that there exists a finite dimensional
$\U$-module $V$ which generates $M$ over $A$. Thus $\Hom_A(M, N)$ is
isomorphic to a submodule of $V^*\otimes_\Bbbk N$ as a $\U$-module,
which is obviously locally finite and thus semi-simple over $\U$.

\begin{proposition}\label{splitting}(Splitting lemma)
Consider a short exact sequence
\begin{eqnarray}\label{splitsequence}
0\longrightarrow M' \longrightarrow M
\stackrel{p}{\longrightarrow} M'' \longrightarrow 0
\end{eqnarray}
in $A$-$\U$-{\bf mod}, where $M''$ is an object of $\cM(A, U)$. If
the exact sequence is $A$-split, then it is also split as an exact
sequence of $A$-$\U$-modules.
\end{proposition}
\begin{proof}
Since the above sequence is $A$-split, there is an  
$A$-module isomorphism 
$$M\overset{\sim}{\lr} M'\oplus M''.
$$ 
Therefore $\Hom_A(M'', M)\stackrel{p\circ -}{\longrightarrow} \Hom_A(M'',
M'')\longrightarrow 0$
is exact. This is an exact sequence of $\U$-modules as the map $p\circ-$ is
clearly $\U$-linear.  Since $M''$ is an object of $\cM(A, U)$ and $M$ is
locally $\U$-finite,  both hom-spaces are
semi-simple $\U$-modules. Hence we have the exact sequence of
$\U$-invariants $\Hom_A(M'', M)^\U\stackrel{p\circ
-}{\longrightarrow} \Hom_A(M'', M'')^\U\longrightarrow 0$.
Note that for any $A$-$U$-modules $N$ and $N'$,
an element $f$ of $\Hom_A(N, N')$ belongs to
$\Hom_\text{$A$-$\U$}(N, N')$ if and only if $x f = \epsilon(x) f$ for all
$x\in\U$. Therefore, we have the exact sequence
\[
\Hom_\text{$A$-$\U$}(M'', M)\stackrel{p\circ -}{\longrightarrow}
\Hom_\text{$A$-$\U$}(M'', M'')\longrightarrow 0.
\]
Now any element of the pre-image of $\id_{M''}$ splits the exact sequence
\eqref{splitsequence} of $A$-$\U$-modules.
\end{proof}

Given a finite dimensional $\U$-module $V$, we define the free
$A$-module $V_A=A\otimes_\Bbbk V$ with the obvious $A$-action (given
by left multiplication). We also define a $\U$-action on it by
\[
x(a\otimes v)=\sum_{(x)} x_{(1)}\cdot a \otimes x_{(2)} v
\]
for all $a\in A$, $v\in V$ and $x\in\U$. These two actions are
easily be shown to be compatible. Call $V_A$ a {\em free}
$A$-$\U$-module of finite rank. Since the module algebra $A$ is
locally $\U$-finite, $V_A$ is also locally $\U$-finite, and
hence belongs to $\cM(A, \U)$.

\begin{lemma}\label{free-module}
For each object $M$ of $\cM(A, \U)$, there exists an exact sequence $V_A\longrightarrow
M\longrightarrow 0$ in $\cM(A, \U)$, where $V_A$ is a free
$A$-$\U$-module of finite rank.
\end{lemma}
\begin{proof}
Given an object $M$ of $\cM(A, \U)$, we may choose any finite set of
generators for it as an $A$-module, and consider the $\U$-module $V$
generated by this set. Then $V$ is finite dimensional because
of the local $\U$-finiteness of $M$. We have the obvious surjective
$A$-$\U$-module map $V_A\longrightarrow M$, $a\otimes v\mapsto a v$.
Since $A$ is locally $\U$-finite, a free $A$-$\U$-module of finite
rank is locally $\U$-finite, thus the exact sequence
$V_A\longrightarrow M\longrightarrow 0$
is in $\cM(A, \U)$.
\end{proof}

As a corollary of Proposition \ref{splitting}, we have the following result.
\begin{corollary}\label{projective}
Let $A$ be a locally finite $\U$-module algebra.
For any object $P$ of $\cM(A,\U)$, the following conditions are equivalent :
\begin{enumerate}
\item $P$ is projective as an $A$-module;
\item $P$ is a projective object of $A$-$\U$-{\bf mod};
\item $P$ is a direct summand of some free $A$-$\U$-module
$V_A=A\otimes_\Bbbk V$ of finite
rank.
\end{enumerate}
\end{corollary}
\begin{proof}
Assume that $P$ is $A$-projective. Then given any exact sequence
$M\to N\to 0$ in $A$-$\U$-{\bf mod}, we have the exact sequence
\[
\Hom_A(P, M)\longrightarrow\Hom_A(P, N)\longrightarrow 0
\]
of $\U$-modules. Since both $\Hom_A(P, M)$ and $\Hom_A(P, N)$ are
semi-simple as $\U$-modules, this leads to the following exact
sequence of $\Bbbk$-vector spaces
\[
\Hom_{A-\U}(P, M)\longrightarrow\Hom_{A-\U}(P, N)\longrightarrow 0.
\]
This proves that (1) implies (2).

Now assume that (2) holds. By Lemma \ref{free-module},
there exists a free $A$-$\U$-module $V_A=A\otimes_\Bbbk V$
of finite rank and an exact sequence
$V_A\stackrel{p}{\longrightarrow} P \longrightarrow 0$ of $A$-$\U$-modules.
It follows from (2) that the
identity map of $P$ factors through $p$. Hence the exact
sequence splits in $\cM(A, \U)$. This proves that
(2) implies (3).

It is evident that (3) implies (1), and the
proposition follows.
\end{proof}

Recall that an algebra $A$ is called {\em left regular} if it is left
Noetherian and every finitely generated left $A$-module has a finite
resolution by finitely generated projective left $A$-modules. 

\begin{proposition}\label{regular-alg}
Let $A$ be a locally finite module algebra over $\U$. Assume that
$A$ is left regular. Then every object $M$ in $\cM(A, \U)$
admits a finite $\cP(A, \U)$-resolution.
\end{proposition}
\begin{proof}
Let $M$ be an object in $\cM(A, \U)$. By Lemma \ref{free-module},
there exists an exact sequence $V_{0, A}
\stackrel{p_0}{\longrightarrow} M\longrightarrow 0$ in $\cM(A, \U)$,
where $V_{0, A}=A\otimes V_0$ is a free $A$-$\U$-module of finite
rank. Since $Ker(p_0)$ also belongs to $\cM(A, \U)$ because $A$ is
left Noetherian, we may apply the same considerations to it, and
inductively we obtain an $A$-free resolution $ \dots\longrightarrow
V_{1, A} \longrightarrow V_{0, A} \longrightarrow M \longrightarrow
0 $ in $\cM(A, \U)$ for $M$. Let $d$ be the projective dimension of
$M$, which is finite because $A$ is regular. It follows from
standard facts in homological algebra (see, e.g., \cite[Lemma
4.1.6]{W}) that the kernel $P$ of the map $V_{d-1, A}
\longrightarrow V_{d-2, A}$ is $A$-projective, hence belongs to
$\cP(A, \U)$ by Corollary \ref{projective}. Thus we arrive at the
$\cP(A, \U)$-resolution
\[
0\longrightarrow P\longrightarrow V_{d-1, A}\longrightarrow
\dots\longrightarrow V_{1, A} \longrightarrow V_{0, A}
\longrightarrow M \longrightarrow 0.
\]
This completes the proof of the proposition.
\end{proof}

For left regular module algebra, we have the following result,
which is an analogue of \cite[Theorem 2.3]{BH2}.
\begin{proposition}\label{KequalK}
Assume that  the locally finite $\U$-module algebra $A$ is left regular.
Then there exist the isomorphisms
\[
 K^\U_i(A)\stackrel{\sim}{\longrightarrow} K_i(\cM(A, \U)), \quad i=0, 1, 2,
\dots.
\]
\end{proposition}
\begin{proof}
Since $A$ is left regular, it must be left Noetherian. Thus $\cM(A,
\U)$ is an abelian category, which has the natural exact structure
consisting of all the short exact sequences.  In view of Proposition
\ref{regular-alg}, the embedding $\cP(A, \U)\subset \cM(A, \U)$
satisfies the conditions of Quillen's Resolution Theorem
\cite[Theorem 4.6]{Sr}. The statement follows.
\end{proof}

%
%

\section{Equivariant K-theory of filtered module algebras}\label{filtered-algebras}

In this section we develop properties of the equivariant $K$-theory
of filtered module algebras over quantum groups. The main results 
here are Theorem \ref{graded} and Theorem \ref{filtered}, which are
quantum analogues of \cite[Theorem 3.2, Theorem 4.1]{BH2}.
The proofs of these theorems are adapted from
\cite[\S 3, \S 4]{BH2} and \cite[\S 6]{Q}.

\subsection{Graded module algebras}
Let $S=\oplus_{n=0}^\infty S_n$ be a graded,
locally finite $\U$-module algebra. We assume that the $\U$-action
preserves the grading of $S$, that is, each $S_n$ is stable under
the $\U$-action. Then $A:=S_0$, is a subalgebra of $S$. Set
$S_+=\oplus_{n>0} S_n$. Then $A$ may be identified with $S/S_+$. We
shall consider positively graded $\U$-equivariant $S$-modules, in
which the $\U$-action preserves the grading. We continue to assume
that all modules are locally $\U$-finite.

For (such) a graded $S$-$\U$-module $N=\oplus_{i=0}^\infty N_i$, we set
\[ T_i(N) = \Tor_i^S(A, N), \quad i\ge 0. \]
These spaces have a natural $S$-$\U$-module structure, which 
may be seen as follows. Take a sequence of graded $\U$-equivariant 
$S$-modules which
form a graded $S$-projective resolution for $N$;  this may be
done by taking, e.g., the usual normalised bar
resolution for $N$, which is a sequence of graded $\U$-equivariant 
$S$-modules as
can be easily seen by examining the explicit definition of the differential.
By Lemma \ref{tensor-functors}, the Tor-groups
$T_i(N)$ computed from such a resolution are $S$-$\U$-modules. In
particular, we have $T_0(N)=N/S_+ N$ with the natural
$S$-$\U$-module structure.

Let us introduce an increasing filtration $0=F_{-1}N\subset
F_0N\subset F_1N\subset\dots$ of $N$ by graded submodules,
by taking $F_pN=\sum_{i\le p} SN_i$.
Then $T_0(F_pN)_n=0$ if $n>p$ and $T_0(F_pN)_n=T_0(N)_n$ if
$n\le p$. There is also a natural $S$-$\U$-module surjection
\begin{eqnarray}\label{Tor-lemma}
S(-p)\otimes_A T_0(N)_p \longrightarrow F_pN/F_{p-1}N,
\end{eqnarray}
where $S(-p)$ is $S$ with the grading shifted by $p$, that is,
$S(-p)_n=S_{n-p}$.

\begin{remark}\label{Tor-rem}
By \cite[Lemma 1, p.117]{Q}, if $T_1(N)=0$ and $\Tor_i^A(S,
T_0(N))=0$ for all $i>0$, then the maps \eqref{Tor-lemma} are
isomorphisms.
\end{remark}

Let $\Mgr(S, \U)$ be the additive category of finitely
$S$-generated, positively graded, and locally $\U$-finite
$S$-$\U$-modules. If we assume that $S$ is left Noetherian,
then $\Mgr(S, \U)$ is abelian, and hence is an exact
category. Its $K$-groups are naturally $\Z[t]$-modules with $t$
acting as the translation functor $N\to N(-1)$.

If $S$ is $A$-flat, then every $A$-projective
resolution $P_{\bullet}\to V$ of $V$ in $\cM(A, \U)$ gives rise to an
$S$-projective resolution $S\otimes_A P_{\bullet}\to S\otimes_A V$.
Hence we have an exact functor $(S\otimes_A -): \cM(A, \U) \to
\Mgr(S, \U)$, which induces homomorphisms
\begin{eqnarray}\label{K-maps}
(S\otimes_A - )_*: K_i(\cM(A, \U)) \longrightarrow K_i(\Mgr(S, \U))
\end{eqnarray}
of $K$-groups.

\begin{theorem}\label{graded}
Assume that $S$ is left Noetherian and $A$-flat. If $A=S/S_+$ 
has finite projective dimension as an
$S$-module, then \eqref{K-maps}
extends to a $\Z[t]$-module isomorphism
\begin{eqnarray}\label{isomorphism}
\Z[t]\otimes_\Z K_i(\cM(A, \U)) \longrightarrow K_i(\Mgr(S, \U)),
\quad \text{for each $i$}.
\end{eqnarray}
\end{theorem}
\begin{proof} We adapt the proofs of \cite[Theorem 6]{Q} and
\cite[Theorem 3.2]{BH2} to the present setting. Let $\cM_p$ denote
the full subcategory of $\Mgr(S, \U)$ with objects $N$ such that
$T_i(N)=0$ for all $i>p$. If the projective dimension of the
$S$-module $A$ is $d$, then
$\cM_0\subset\cM_1\subset\dots\subset\cM_d=\Mgr(S, \U)$. For $N$ in
$\cM_p$, Lemma \ref{free-module} gives a surjective $S$-$\U$-map
$V_S=S\otimes_{\Bbbk} V\twoheadrightarrow N$, $s\otimes v\mapsto s
v$, where $V$ is a finite dimensional $\U$-submodule of $N$ which
generates $N$ iself over $S$. Then the kernel $N'$ of the surjection
belongs to $\Mgr(S, \U)$ since $S$ is left Noetherian. Because $V_S$
is a free $S$-module, $T_i(V_S)=0$ for all $i>0$. Hence the long
exact sequence of $\Tor$ groups arising from the short exact
sequence $0\to N'\to V_S \to N\to 0$ yields $T_i(N) \cong
T_{i-1}(N')$ for all $i>0$. This then implies that $N'$ belongs to
$\cM_{p-1}$.

Therefore, the inclusion $\cM_p\subset\cM_{p+1}$ for every $p\ge 0$
satisfies the conditions of the Resolution Theorem
\cite[Theorem 4.6]{Sr}, hence is a homotopy equivalence. This leads
to the homotopy equivalence $\cM_0\subset \cM_d=\Mgr(S, \U)$, which
induces the isomorphisms
\begin{eqnarray}\label{isom-grd}
K_i(\cM_0) \stackrel{\sim}{\longrightarrow}K_i(\Mgr(S, \U)), \quad
\text{for all }  i=0, 1, \dots.
\end{eqnarray}

Let $V$ be an object in $\cM(A, \U)$, and let $P_\bullet\to V$ be an
$A$-projective resolution. Since $S$ is $A$-flat, $S\otimes_A
P_\bullet\to S\otimes_A V$ is an $S$-projective resolution, and
hence $T_i(S\otimes_A V) =0$, for all $i>0$. Thus for any $V$ in
$\cM(A, \U)$, $S\otimes_A V$ belongs to $\cM_0$. Therefore, if 
$\cM_{0, n}$ be the full subcategory of $\cM_0$ whose objects are
modules $M$ such that $M=F_n M$, then we have an exact functor
\[
\begin{aligned}
&b: \cM(A, \U)^{n+1} \longrightarrow \cM_{0, n}, \\
&(V_0, V_1, \dots, V_n) \mapsto \oplus_{p=0}^n S(-p)\otimes_A V_p.
\end{aligned}
\]
This induces homomorphisms
\[
b_*: K_i(\cM(A, \U))^{n+1} \longrightarrow K_i(\cM_{0, n}).
\]
Since $T_0$ is exact on $\cM_0$, we also have an exact functor
\[
\begin{aligned}
&c: \cM_{0, n}\longrightarrow \cM(A, \U)^{n+1} , \\
& M\mapsto (T_0(M)_0, T_0(M)_1, \dots, T_0(M)_n),
\end{aligned}
\]
and homomorphisms
\[
c_*: K_i(\cM_{0, n}) \longrightarrow K_i(\cM(A, \U))^{n+1}.
\]

Note that $c\circ b$ is equivalent to the identity functor, thus
$c_*\circ b_* = \id$.

On the other hand, any $M$ in $\cM_{0, n}$ has a filtration
\[
0=F_{-1} M\subset F_0 M \subset F_1 M \subset \dots \subset F_n M=M.
\]
Because of the $A$-flat nature of $S$, Remark \ref{Tor-rem} applies and
we have $\frac{F_p M}{F_{p-1}M} = S(-p)\otimes_A T_0(M)_p$. Clearly
each functor $\frac{F_p}{F_{p-1}} $ is exact. It follows from the
additivity of characteristic filtrations \cite[Corollary 2,
p.107]{Q} \cite[Corollary 4.4]{Sr} that $
\sum_{p=0}^n\left(\frac{F_p}{F_{p-1}}\right)_*=(F_n)_*=\id$. Now
observe that $(b\circ
c)_*=\sum_{p=0}^n\left(\frac{F_p}{F_{p-1}}\right)_*$, hence $
b_*\circ c_*= \id$.

Passing to the limit $n\to \infty$ we have the following isomorphism
for each $i$:
\[ \Z[t]\otimes_\Z K_i(\cM(A, \U)) \longrightarrow K_i(\cM_0). \]
Using the isomorphisms \eqref{isom-grd}, we arrive at the desired
result.
\end{proof}

\subsection{Filtered module algebras}
Let $S$ be a locally finite $\U$-module algebra.
Assume that $S$ has an ascending filtration $0\subset F_{-1}S \subset
F_0S\subset F_1S \subset \dots$ such that $1\in F_0S$, $\cup_p
F_pS=S$ and $F_pS F_qS\subset F_{p+q}S$. We assume that the
filtration is preserved by the $\U$-action. Let
\[
{\overline S} = gr(S) := \oplus_{p\ge 0} {\overline S}_p \quad
\text{with} \quad {\overline S}_p:=F_pS/F_{p-1}S,
\]
and set $A=F_0 S$ and ${\overline S}_+ = \oplus_{p>0}{\overline
S}_p$.

\begin{theorem}\label{filtered}
Assume that ${\overline S}$ is left Noetherian and $A$-flat. If $A$
($= {\overline S}/{\overline S}_+$) has a finite projective
${\overline S}$-resolution, then for $i= 0, 1,
2, \dots$ there exist isomorphisms
\begin{eqnarray}
K_i(\cM(A, \U)) \stackrel{\sim}{\longrightarrow} K_i(\cM(S, \U)).
\quad  
\end{eqnarray}
If furthermore $A$ is regular, then $S$ is regular and there exist
isomorphisms
\begin{eqnarray}
K^\U_i(A) \stackrel{\sim}{\longrightarrow} K^\U_i(S). \quad 
\end{eqnarray}
\end{theorem}

\begin{remark}
If $\fg=0$ and $\U$ is generated by the identity, Theorem
\ref{filtered} reduces to a slightly weaker version of \cite[Theorem
7]{Q}. See also the question raised in \cite[p.118]{Q} (immediately
below \cite[Theorem 6]{Q}).
\end{remark}

In order to prove the theorem, we need some preliminaries.
Let $z$ be an indeterminate, and consider the graded
algebra
\[S' = \oplus_{p\ge 0} (F_pS)z^p,\]
where $z$ is central in $S'$ and has degree $1$. This is a
subalgebra of $S[z]$. We endow $S'$ with a $\U$-action 
by specifying that $z$ is $\U$-invariant. 
This turns $S'$ into a $\U$-module algebra. Let
$S'_+=\oplus_{p>0} (F_pS)z^p$, then $A=S'/S'_+$. Note also that
$\overline{S} = S'/zS'$.

The next result does not involve the $\U$-action.
\begin{lemma}\label{S-primed}
Assume that ${\overline S}$ is left Noetherian and $A$-flat.
\begin{enumerate}
\item Then $S'$ is left Noetherian and $A$-flat.

\item If it is further assumed that $A(=S/S_+)$ has a finite
projective ${\overline S}$-resolution, then $A(=S'/S'_+)$ also has a
finite projective $S'$-resolution.
\end{enumerate}
\end{lemma}
\begin{proof}
Filter $S'$ by letting $F_p S'$ consist of polynomials in $z$ with
coefficients in $F_p S$. Then the associated graded algebra of $S'$
is given by
\[
\rm{gr}(S')=\oplus_{p\ge 0}\frac{F_p S'}{F_{p-1}S'}
= \oplus_{p\ge 0} \overline{S}_p z^p.
\]
Since $\overline{S}$ is left Noetherian, so also is
$\overline{S}[z]$.  This then implies that $S'$ is left Noetherian
(see, e.g., \cite[Lemma 3.(i), p.119]{Q}).

Given that $\overline{S}$ is $A$-flat, so also is every graded
component $\overline{S}_p$, and in particular,
$F_0S=\overline{S}_0$. Corresponding to any short exact sequence
$0\longrightarrow V' \longrightarrow V \longrightarrow V''
\longrightarrow 0$ of $A$-modules, we have the commutative diagram
\[
\begin{array}{c c  c c c c c c c}
&              &      0           &
 &       0    &                    &    0            &     \\
 &                & \downarrow &
 & \downarrow &                        & \downarrow &  \\
0&\longrightarrow &F_{p-1}S\otimes_A V'&\longrightarrow&F_{p-1}S\otimes_A V
&\longrightarrow&F_{p-1}S\otimes_A V''&\longrightarrow & 0\\
 &     & \downarrow &     & \downarrow &     & \downarrow &  \\
0& \longrightarrow &F_{p}S\otimes_A V'&\longrightarrow & F_{p}S\otimes_A V
& \longrightarrow&F_{p}S\otimes_A V'' & \longrightarrow & 0\\
&    & \downarrow &   & \downarrow &    & \downarrow &  \\
0& \longrightarrow &\overline{S}_{p}\otimes_A V'& \longrightarrow
& \overline{S}_{p}\otimes_A V& \longrightarrow&\overline{S}_{p}
\otimes_A V''& \longrightarrow & 0\\
&    &\downarrow &    & \downarrow &     & \downarrow &  \\
&    &      0    &     &       0    &      &    0     &
\end{array},
\]
where the columns are obviously exact. Since the composition of the
maps in the middle row is zero, exactness of the top and bottom rows
will imply the exactness of the middle one. Thus by induction on
$p$, we can show that $F_p S$ is $A$-flat for all $p$. This then
immediately leads to the $A$-flatness of $S'$.

To prove the second part of the lemma, we note that
$\overline{S}=S'/zS'$ leads to $pd_{S'}(\overline{S})=1$, where
$pd_{R}(M)$ denotes the projective dimension of the left $R$-module
$M$. Thus by the Change Rings Theorem \cite[Theorem 4.3.1]{W},
$pd_{S'}(A) \le pd_{\overline{S}}(A)+pd_{S'}(\overline{S})<\infty$.
This completes the proof of the lemma.
\end{proof}

Since $S=S'/(1-z)S'$, it follows from Lemma \ref{S-primed}.(1) that
$S$ is left Noetherian. Hence $K_i(\cM(S, \U))$ are defined.

The proof of the following lemma requires both the Localisation
Theorem \cite[Theorem 5, p.105]{Q} and Devissage Theorem
\cite[Theorem 4, p.104]{Q}.

\begin{lemma}\label{long-sequence}
There exists the following long exact sequence of K-groups:
\[
\begin{aligned}
\cdots &\longrightarrow  K_1(\Mgr(\overline{S}, \U))
\longrightarrow K_1(\Mgr(S', \U))
\longrightarrow K_1(\cM(S, \U))\\
& \longrightarrow
K_0(\Mgr(\overline{S}, \U))\longrightarrow K_0(\Mgr(S', \U))
\longrightarrow K_0(\cM(S, \U))\longrightarrow 0.
\end{aligned}
\]
\end{lemma}

\begin{proof} Recall the crucial facts that $z$ is $\U$-invariant
and is also central in $S'$. Let $\cN$ be the full subcategory of
$\Mgr(S', \U)$ consisting of modules killed by some power of $z$.
This is a Serre subcategory, so we may define the quotient
category $\Mgr(S', \U)/\cN$. A more concrete way to view this
construction is as follows. Let $S'[z^{-1}]$ be the localisation of
$S'$ at $z^{-1}$. Then the localisation functor $\Mgr(S',
\U)\longrightarrow \Mgr(S'[z^{-1}], \U)$ annihilates precisely the
modules in $\cN$, and $\Mgr(S', \U)/\cN$ is equivalent to
$\Mgr(S'[z^{-1}], \U)$.

For any object $M$ in  $\Mgr(S'[z^{-1}], \U)$, $z^{-1}$ acts as an
isomorphism. Hence $M$ is uniquely determined by its degree $0$
component. This leads to an equivalence of categories
$\Mgr(S'[z^{-1}], \U)\cong \cM(S, \U)$. Denote by $j: \Mgr(S', \U)
\longrightarrow\cM(S, \U)$ the composition of this equivalence with
the localisation functor. Then $j: M\mapsto M/(1-z)M$ for all $M$ in
$\Mgr(S', \U)$.

Now using Quillen's Localisation Theorem
\cite[Theorem 5, p.105]{Q}, we obtain 
(cf. {\it op. cit.} p. 115) a long exact sequence of K-groups:
$$
\begin{aligned}
\cdots \longrightarrow K_i(\cN)&\longrightarrow K_i(\Mgr(S', \U))
\longrightarrow K_i(\cM(S, \U))\longrightarrow
K_{i-1}(\cN)\longrightarrow\cdots\\ &\cdots \longrightarrow
 K_0(\cM(S, \U)) \longrightarrow 0.\\
\end{aligned}
$$

Since $\overline{S} = S'/zS'$, we see that $\Mgr(\overline{S}, \U)$
is a full subcategory of $\Mgr(S', \U)$ with objects the modules
annihilated by $z$. Hence we have an inclusion $\Mgr(\overline{S},
\U)\subset \cN$. As any object $N$ of $\cN$ is annihilated by $z^k$
for some $k$, we have the filtration $0=z^k N \subset z^{k-1} N
\subset \dots \subset z N \subset N$ with $\frac{z^i N}{z^{i-1} N}$
obviously annihilated by $z$ for each $i$. Hence the Devissage
Theorem \cite[Theorem 4, p.112]{Q} applies to this situation,
inducing isomorphisms of K-groups $K_i(\cN) \cong
K_i(\Mgr(\overline{S}, \U))$ for all $i\ge 0$. The 
lemma follows.
\end{proof}

\begin{remark}\label{i-j-maps}
Let $i: \Mgr(\overline{S}, \U) \longrightarrow \Mgr(S', \U)$ denote
the composition of the inclusions  $\Mgr(\overline{S}, \U)\subset
\cN\subset\Mgr(S', \U)$. From the proof of the lemma we see that
maps beside the connecting homomorphism in the long exact sequence
are $i_*$ and $j_*$ as indicated below
\[\cdots \longrightarrow K_i(\Mgr(\overline{S},
\U))\stackrel{i_*}{\longrightarrow} K_i(\Mgr(S', \U))
\stackrel{j_*}{\longrightarrow} K_i(\cM(S,
\U))\longrightarrow\dots.\]
\end{remark}

\begin{proof}[Proof of Theorem \ref{filtered}]
Since $\overline{S}$ satisfies the conditions of Theorem
\ref{graded}, by Lemma \ref{S-primed}, $S'$ also satisfies the
conditions of the theorem. Therefore, we have the $\Z[t]$-module
isomorphisms
\begin{eqnarray}\label{K-is}
\begin{aligned}
\Z[t]\otimes_\Z K_i(\cM(A, \U)) \longrightarrow
K_i(\Mgr(\overline{S}, \U)),
&\quad& 1\otimes g \mapsto (\overline{S}\otimes_A -)_* g, \\
\Z[t]\otimes_\Z K_i(\cM(A, \U)) \longrightarrow K_i(\Mgr(S', \U)),
&\quad& 1\otimes g \mapsto (S'\otimes_A -)_* g.
\end{aligned}
\end{eqnarray}
Let us now describe the map $i_*$ in Remark \ref{i-j-maps} more
explicitly by finding the map $\delta$ which renders the following
diagram commutative:
\begin{center}
\begin{picture}(280, 80)(0, 60)
 \put(20, 120){$K_i(\Mgr(\overline{S}, \U))$}
 \put(110, 123){\vector(1,0){50}}
 \put(120, 130){$i_*$}
 \put(175, 120){$ K_i(\Mgr(S', \U)) $}

 \put(50, 80){\vector(0, 1){35}}
 \put(35, 95){$\cong$}
 \put(205, 80){\vector(0, 1){35}}
 \put(210, 95){$\cong$}

 \put(0, 65){$\Z[t]\otimes_\Z K_i(\cM(A, \U))$}
 \put(110,70){\vector(1,0){50}}
 \put(130,75){$\delta$}
 \put(175, 65){$\Z[t]\otimes_\Z K_i(\cM(A, \U))$.}
\end{picture}
\end{center}
Here the vertical isomorphisms are given by \eqref{K-is}. For every
$M$ in $\cM(A, \U)$, we have an exact sequence
\[ 0\longrightarrow S'(-1)\otimes_A M
\stackrel{z}{\longrightarrow} S'\otimes_A M \longrightarrow
\overline{S}\otimes_A M\longrightarrow 0
\]
in $\Mgr(S', \U)$, where $\overline{S}\otimes_A M$ is in
$\Mgr(\overline{S}, \U)$ but is regarded as an object of $\Mgr(S', \U)$ via
the inclusion $i$. Therefore the composition
$i\circ(\overline{S}\otimes_A -)$ of functors $\cM(A, \U)
\stackrel{\overline{S}\otimes_A -}{\longrightarrow}\cM(\overline{S},
\U)\stackrel{i}{\longrightarrow}\Mgr(S', \U)$ fits into an exact
sequence of functors
\[
0\longrightarrow S'(-1)\otimes_A - \longrightarrow S'\otimes_A -
\longrightarrow i\circ(\overline{S}\otimes_A -)\longrightarrow 0
\]
from $\cM(A, \U)$ to $\Mgr(S', \U)$. By Corollary 1 of Theorem 2 in
 \cite[p.106]{Q} , we have
\[
i_*\circ(\overline{S}\otimes_A -)_* = (S'\otimes_A - )_*
-(S'(-1)\otimes_A -)_* = (1-t)(S'\otimes_A - )_*.
\]
From this formula it is evident that the map $\delta$ is
multiplication by $1-t$, which is injective with cokernel
$K_i(\cM(A, \U))$.

Therefore  $i_*$ is injective with cokernel isomorphic to
$K_i(\cM(A, \U))$. Using this information in  the long exact
sequence of Lemma \ref{long-sequence}, we deduce
that the connecting morphism is zero, and $K_i(\cM(S,
\U))$ is isomorphic to the cokernel of $i_*$. Hence 
the composition of functors $\cM(A, \U) \longrightarrow
\Mgr(S', \U) \stackrel{j}{\longrightarrow} \cM(S, \U)$,  $M\mapsto
S'\otimes_A M\mapsto S\otimes_A M$ (where $S=S'/(1-z)S'$) induces an
isomorphism
\[
K_i(\cM(A, \U)) \longrightarrow K_i(\Mgr(S', \U))
\stackrel{j_*}{\longrightarrow} K_i(\cM(S, \U)).
\]
This proves the first assertion of the theorem.

Given the conditions that $\overline{S}$ is left Noetherian
and $A$ has finite projective dimension as a left
$\overline{S}$-module, \cite[Lemma 4, p.120]{Q} applies and hence
the regularity of $A$ implies the regularity of $S$. Thus it follows
from Proposition \ref{KequalK} that $K_i(\cM(A, \U))= K^\U_i(A)$ and
$K_i(\cM(S, \U))= K^\U_i(S)$ for all $i$. Now the second part of the
theorem immediately follows from the first part.
\end{proof}

\section{Quantum symmetric algebras}\label{quantum-symmalg}

We now apply results from Section
\ref{filtered-algebras} to compute the equivariant K-groups of a
class of module algebras over quantum groups. We shall refer to
these module algebras as {\em quantum symmetric algebras}; these are
quadratic algebras of Koszul type naturally arising from the
representation theory of quantum groups. The quantised coordinate
ring of affine $n$-space is an example. See \cite{LZZ, Z, BZ}
for other examples.

\subsection{Equivariant K-theory of quantum symmetric algebras}

Let $V$ be a finite dimensional vector space over a field $\Bbbk$,
and denote by $T(V)$ the tensor algebra
over $V$. Given a subset $I$ of $V\otimes_\Bbbk V$, we denote by
$\langle I \rangle$ the two-sided ideal of $T(V)$ generated by $I$.
Define the {\em quadratic algebra}
\[A:=T(V)/\langle I \rangle. \]
We shall also use the notation $\Bbbk\{V, I\}$ for
$A$ to indicate the generating vector space $V$ and
the defining relations of the algebra explicitly. The algebra $A$ is naturally
$\Z_+$-graded since $\langle I \rangle$ is. We have
$ A=\bigoplus_{i=0}^\infty A_i, $
with $A_0=\Bbbk$ and $A_1=V$.

We shall say that a quadratic algebra $A=\Bbbk\{V, I\}$ is of  {\em
PBW type} if there exists a basis $\{v_i\mid i=1, 2, \dots, d\}$ of
$V$ such that the elements $v^{\bf a} := v_1^{a_1} v_2^{a_2}\cdots
v_d^{a_d}$, with ${\bf a}:=(a_1, a_2, \dots, a_d)\in \Z_+^d$, form a
basis (called the PBW basis) of $A$.

Let $\Bbbk$ be the field $\C(q)$.  


\begin{lemma}\label{TV-I}
Let $V$ be a finite dimensional module of type $(1, \dots, 1)$ over
a quantum group $\U$. Let $I\subset V\otimes_\Bbbk V$ be a
$\U$-submodule. Then the quadratic algebra $\Bbbk\{V,
I\}=T(V)/\langle I \rangle$ is a 
$\U$-module algebra.
\end{lemma}
\begin{proof}
The tensor algebra $T(V)$ has a natural $\U$-module algebra
structure, with the $\U$-action defined by using the
co-multiplication. Since $I$ is a $\U$-submodule of $V\otimes V$, so
also is the two-sided ideal $\langle I\rangle$. Hence
$A=T(V)/\langle I\rangle$ is a $\U$-module algebra.
\end{proof}

\begin{definition}
We call a $\U$-module algebra $A=\Bbbk\{V, I\}$ 
of the type defined in Lemma
\ref{TV-I} a {\em quantum symmetric algebra} of the finite
dimensional $\U$-module $V$ if it admits a PBW basis.
\end{definition}

If $A=\Bbbk\{V, I\}$ admits a PBW basis, it is sometimes referred to as
`flat'. 

The next theorem is our main result concerning the equivariant K-theory of 
quantum symmetric algebras. 

\begin{theorem}\label{q-symmalg}
Let $A=\Bbbk\{V, I\}$ be a quantum symmetric algebra of a finite
dimensional module $V$ over the quantum group $\U$. Assume that $A$
is  left  Noetherian, then
\[
K^\U_i(A)=K_i(\Umod), \quad \text{for all $i=0, 1, \dots$},
\]
where $\Umod$ is the category of finite dimensional left
$\U$-modules of type-$(1, \dots, 1)$.
\end{theorem}

We shall prove this result using Theorem \ref{filtered}
of Section \ref{q-symmalg-proof}.
In order to do that, we need some results from
the theory of Koszul algebras, which we now discuss.

\subsection{Quadratic algebras and Koszul complexes}\label{Koszul-complex}\label{PBW}

In this subsection, $\Bbbk$ may be any field. Let $V$ be a finite
dimensional vector space. Given a subspace $I$ of $V\otimes_\Bbbk
V$, we have the corresponding quadratic algebra $A=\Bbbk\{V, I\}$.
Let $V^*$ be the dual vector space of $V$, and define
\[
I^\bot:=\{\alpha\in V^*\otimes_\Bbbk V^* \mid \alpha(w)=0 \text{ for
all $w\in I$}\}.
\]
This definition implicitly uses the canonical isomorphism 
$(V\otimes_\Bbbk V)^*\cong V^*\otimes_\Bbbk V^*$. Henceforth 
$\otimes$ will denote $\otimes_\Bbbk$. Let $\langle I^\bot \rangle$ be
the two-sided ideal of the tensor algebra $T(V^*)$ over $V^*$
generated by $I^\bot$. We may define the quadratic algebra
\[ A^!:=T(V^*)/\langle I^\bot \rangle, \]
which is referred to as the {\em dual quadratic algebra} of $A$. We endow
$T(V^*)$ with a $\Z_-:=-\Z_+$ grading with $V^*$ having degree $-1$. Then
$\langle I^\bot \rangle$ is a two-sided graded ideal, and $A^!$ is
$\Z_-$-graded with $V^*$ having degree $-1$.

Let $A=\Bbbk\{V, I\}$ and $A^!=\Bbbk\{V^*, I^\bot\}$ be dual
quadratic algebras. Then $A^!\otimes A$ has a natural algebra
structure such that the subalgebras $A$ and $A^!$ commute.
Let $z$ be the image of the identity element of  $\Hom_\Bbbk(V,
V)$ in $V^*\otimes V$ under the natural isomorphism, and let $e_A$
be its image in $A^!\otimes A$.
It is easily verified that $e_A^2=0$.

We regard $A$ as a right $A$-module, and $A^!$ as a left
$A^!$-module. Then the graded dual ${A^!}^* = \bigoplus_{i\in \Z_+}
{A^!}^*_i$ of $A^!$ with ${A^!}^*_i=(A^!_{-i})^*$ has a natural
right $A^!$-module structure. Hence ${A^!}^*\otimes A$ is a right
$A^!\otimes A$-module. The action of $e_A$ defines a differential on
${A^!}^*\otimes A$, yielding the Koszul complex of $A$:
\begin{eqnarray}\label{resolution}
\cdots \longrightarrow {A^!_2}^* \otimes A  \longrightarrow{A^!_1}^*
\otimes A \longrightarrow A.
\end{eqnarray}

\begin{remark}
We may also regard $A$ as a left $A$-module and $A^!$ as a right
$A^!$-module. Then we have a Koszul complex of left $A$-modules:
\begin{eqnarray}\label{resolution-2}
\cdots \longrightarrow   A \otimes {A^!_2}^* \longrightarrow A
\otimes {A^!_1}^* \longrightarrow A,
\end{eqnarray}
where the differential is $\tilde{e}_A=\sum_{i=1}^d v_i\otimes
\bar{v}_i$.
\end{remark}

For any pair of graded left $A$-modules $M$ and $N$, one may compute
the extension spaces $\Ext^\bullet(M, N):=\bigoplus_i \Ext^i_A(M,
N)$ defined as the right derived functor of the graded homomorphism
functor $\Hom_A(M, N)$. Under the Yoneda product,
$\Ext^\bullet(\Bbbk, \Bbbk)$ forms a graded algebra.
A quadratic algebra $A=\Bbbk\{V, I\}$ is called {\em Koszul} if
$\Ext^\bullet(\Bbbk, \Bbbk)\cong A^!$ as graded algebras.
If $A$ is Koszul, so is also $A^!$.

A key property of a Koszul algebra is that the Koszul complex
\eqref{resolution} of $A$ is a graded free resolution of the base
field $\Bbbk$ regarded as a right $A$-module; that is, the 
complex
\begin{eqnarray}
\cdots \longrightarrow {A^!_2}^* \otimes A  \longrightarrow{A^!_1}^*
\otimes A \longrightarrow A\stackrel{\epsilon}{\longrightarrow}
\Bbbk \longrightarrow 0
\end{eqnarray}
is exact. The map $\epsilon$ is the augmentation $A\longrightarrow
A/A_+$, where $A_+=\bigoplus_{i>0}A_i$. For the proof of this fact,
see e.g., \cite[Corollary II.3.2]{PP}. Similarly,
\eqref{resolution-2} leads to a graded free resolution for the base
field $\Bbbk$ regarded as a left $A$-module in this case.

For quadratic algebras of PBW type, we have the following result.

\begin{theorem}\label{KoszulResolution}
Let $A=\Bbbk\{V, I\}$ be a quadratic algebra of PBW type, and denote
by $A^!$ its dual quadratic algebra. Then:
\begin{enumerate}
\item The algebra $A$ is Koszul.
\item Let $d=\dim_{\Bbbk}V$, then
$\dim_{\Bbbk}A_i = \begin{pmatrix}d+i-1 \\ i
\end{pmatrix}$ and
$\dim_{\Bbbk} A^!_{-i} =
\begin{pmatrix}d \\ i
\end{pmatrix}$ for all $i\ge 0$
\item \label{KoszulResolution2}
The Koszul complexes \eqref{resolution} and \eqref{resolution-2} of
$A$ are graded free resolutions of length $\dim_\Bbbk V$ of the
base field $\Bbbk$.
\end{enumerate}
\end{theorem}
\begin{proof}
The Koszul nature of quadratic algebras of PBW type is a well-known
fact, which was originally established in \cite[Theorem 5.3]{P}.

The dimension of $A_i$ can be easily computed from the PBW basis.
Let $h_A(z) = \sum_{i=0}^\infty z^i \dim_{\Bbbk} A_i$ and
$h_{A^!}(z)= \sum_{i=0}^\infty z^i \dim_{\Bbbk}A^!_{-i}$. Then
$h_A(z) =1/(1-z)^d$. It follows from some general facts on the
Hilbert series of $\Ext^\bullet(\Bbbk, \Bbbk)$ that $ h_A(z)
h_{A^!}(-z)=1$. By part (1), $A$ is Koszul, thus $A^!\cong
\Ext^\bullet(\Bbbk, \Bbbk)$.  Hence $h_{A^!}(z)=(1+z)^d$,
which implies the claimed dimension formula for $A^!_{-i}$.

The Koszul complexes are free resolutions since $A$ is Koszul by
part (1). As $A^!_{-i}=0$ for all $i>d$, the
length of the resolutions is $\dim_\Bbbk V$.
\end{proof}

\begin{remark}
Theorem \ref{KoszulResolution}.(\ref{KoszulResolution2}) will suffice
for the purpose of proving Theorem \ref{usual} and Theorem \ref{q-symmalg}
on $K$-groups.
The results below give a direct proof
that (left) Noetherian quantum symmetric algebras are regular.
\end{remark}

Let $M=\bigoplus_{i\in\Z}M_i$ be a graded module for a quadratic
algebra $A=\Bbbk\{V, I\}$. If $M$ is finitely generated, there
exists some integer $r$ such that $M_i=0$ for all $i<r$. That is, a
finitely generated graded module must be bounded below. Let
\[ \overline{M}:= \frac{A}{A_+}\otimes_A M \cong \Bbbk\otimes_A M. \]
The following result, is a special case of
\cite[Theorem 4.6]{Lam}.
\begin{theorem} Let $P$ be a finitely generated graded projective module over
a quadratic algebra $A=\Bbbk\{V, I\}$. Then $P$ is obtained from
$\overline{P}$ by extension of scalars. That is,
$P \cong A\otimes_{\Bbbk}\overline{P}$. Therefore, all finitely generated
graded projective modules over a quadratic algebra are free.
\end{theorem}

Recall that a complex of graded left $A$-modules
\[
\cdots \longrightarrow C_n \stackrel{\phi_n}{\longrightarrow}
C_{n-1} \stackrel{\phi_{n-1}}{\longrightarrow} C_{n-2}
\longrightarrow \cdots
\]
is called minimal if $\phi_n(C_n)\subseteq A_+C_{n-1}$ for each $n$.
A minimal resolution is defined similarly.

\begin{theorem}\label{syzygy}
Every finitely generated graded left module over a quadratic algebra
$A=\Bbbk\{V, I\}$ of PBW type has a minimal free resolution of
length at most $ \dim V$.
\end{theorem}
\begin{proof}
By \cite[Proposition \S1.4.2.]{PP}, every finitely generated graded
left $A$-module $M$ has a minimal free resolution
\begin{eqnarray}\label{F}
\cF: \cdots \longrightarrow F_2 \longrightarrow F_1 \longrightarrow
F_0 \longrightarrow M \longrightarrow 0.
\end{eqnarray}
Thus it suffices to show that this resolution has finite length.
Let us compute $\Tor^A_\bullet(\Bbbk, M)$ from the complex
$\Bbbk\otimes_A\cF$. By using the minimality of $\cF$, we obtain
\[
\Tor^A_i(\Bbbk, M) = \Bbbk\otimes_A F_i \quad \text{\ for all $i$}.
\]
In particular, $\Tor^A_i(\Bbbk, M) = 0$ if and only if $F_i=0$.

On the other hand, we may also compute $\Tor^A_\bullet(\Bbbk, M)$
from the complex obtained by tensoring (over $A$) the Koszul
resolution of the base field $\Bbbk$ (as a right $A$-module) with
$M$. By part (\ref{KoszulResolution2}) of Theorem
\ref{KoszulResolution}, $\Tor^A_i(\Bbbk, M)=0$ for all $i>\dim V$.
This leads to $F_i=0$ for all $i>\dim V$ in the minimal free
resolution for $M$.
\end{proof}

The following result is a consequence of Theorem \ref{syzygy}.

\begin{theorem}\label{finite-resolution}
Every finitely generated left module over a quadratic algebra
$A=\Bbbk\{V, I\}$ of PBW type has a free resolution of finite
length.
\end{theorem}
\begin{proof}
Let $A^n\stackrel{\phi}{\longrightarrow} A^m\longrightarrow M
\longrightarrow 0$ be an exact sequence of left $A$-modules. We
think of $A^m$ and $A^n$ as consisting of rows with entries from
$A$. Then $\phi$ can be represented as an $n\times m$-matrix
$(\phi_{i j})$ with entries $\phi_{i j}\in A$, and acts on $A^n$ by
matrix multiplication from the right.

Now consider the algebra $T=A[x]$, which consists of polynomials
in $x$ with coefficients in $A$. We stipulate that $x$ commutes with
all elements of $A$. Then $A=T/(1-x)T$. It is easy to see that $T$
is a quadratic algebra which is also Koszul.

Let $r$ be the smallest integer such that every entry $\phi_{i j}$
of the matrix of $\phi$ is contained in $A_0\oplus
A_1\oplus\cdots\oplus A_r$. Then we can write $\phi_{i j} = \phi_{i
j}[0] + \phi_{i j}[1]+\cdots+\phi_{i j}[r]$ with $\phi_{i j}[k]\in
A_k$. Upon replacing the entries $\phi_{i j}$ of the matrix by
$\phi_{i j}(x)=x^r\phi_{i j}[0] + x^{r-1}\phi_{i
j}[1]+\cdots+\phi_{i j}[r]$, we obtain a rectangular matrix
$\tilde\phi=(\tilde\phi_{i j})$ with entries which are homogeneous
elements of $T$ of degree $r$.

Regard $\tilde\phi$ as a left $T$-module homomorphism
$T^n\longrightarrow T^m$, and let $\tilde M=coker\tilde\phi$. Then
$\tilde M$ is a graded $T$-module, and $A\otimes_T \tilde M=M$. By
Theorem \ref{finite-resolution}, we have a free resolution $\cF$ of
$\tilde M$ with finite length. The homology of the complex
$A\otimes_T\cF$ is $\Tor^T_\bullet(A, \tilde M)$.

To compute $\Tor^T_\bullet(A, \tilde M)$, we tensor  with $\tilde M$
the free resolution
\[
0\longrightarrow T \stackrel{1-x}{\longrightarrow} T \longrightarrow
A\longrightarrow 0
\]
of the right $T$-module $A$, obtaining
\[
0\longrightarrow \tilde M\stackrel{1-x}{\longrightarrow} \tilde M.
\]
Since the action of $1-x$ on any graded $T$-module is injective,
the above complex is exact. This shows that
\begin{eqnarray}\label{homology}
\Tor^T_i(A, \tilde M)=0, \quad \text{for all $i\ge 1$}.
\end{eqnarray}

Note that $A\otimes_T\cF$ is also a complex of left $A$-modules,
which are all free except for $A\otimes_T \tilde M=M$. From equation
\eqref{homology} we see that this complex is exact, thus is a free
resolution of finite length for the $A$-module $M$.
\end{proof}

Recall that a left module $E$ over a ring $R$ is called {\it stably free}
if there exists a free left $A$-module $F$ of finite rank  such that
$E\oplus F \cong R^n$ for some finite $n$. Clearly a stably free
module is finitely generated and projective.

It is easy to see that if a projective module admits a free
resolution of finite length, it is stably free. Conversely, a
projective $R$-module $E$ with a free resolution
\[
0\longrightarrow F_n \longrightarrow \cdots \longrightarrow
F_1\longrightarrow F_0\longrightarrow E\longrightarrow 0
\]
of finite length $n$ can be shown to be stably free by a simple
induction on $n$. Indeed, if $n=0$, the claim is obviously true. Let
$E'=Ker(F_0\longrightarrow E)$, then we have the following free
resolution for $E'$:
\[
0\longrightarrow F_n \longrightarrow \cdots \longrightarrow
F_1\longrightarrow E'\longrightarrow 0.
\]
Since the length of the resolution is $n-1$, by the induction
hypothesis, $E'$ is stably free. Hence $E$ is stably free since
$F_0\cong E\oplus E'$.

The next statement is an immediate consequence of Theorem \ref{finite-resolution}.

\begin{corollary}\label{stable-free}
Every finitely generated projective module over a quadratic algebra
$A=\Bbbk\{V, I\}$ of PBW type is stably free.
\end{corollary}

Every quadratic algebra $A=\Bbbk\{V, I\}$ of PBW type is a
$\Z_+$-graded algebra with degree $0$ subalgebra $\Bbbk$.
If the algebra is assumed to be left Noetherian, then Theorem
\ref{finite-resolution} implies that it is regular.

\begin{lemma}\label{regular-qsymm}
A quadratic algebra of PBW type is left regular if it is left
Noetherian. In particular, every left Noetherian quantum symmetric
algebra is left regular.
\end{lemma}

We may now use \cite[Theorem 7]{Q} to compute the usual algebraic
K-groups $K_i(A)$ of $A$. The result is as follows.

\begin{theorem}\label{usual}
Let $A$ be a quadratic algebra of PBW type, and assume that $A$ is
left Noetherian. Then
\[K_i(A)=K_i(\Bbbk), \quad i=0, 1, \dots.\]
\end{theorem}
In particular, $K_0(\Bbbk)=\Z$. This is consistent with Corollary
\ref{stable-free}.
%
%
\subsection{Proof of Theorem \ref{q-symmalg}}\label{q-symmalg-proof}
By Lemma \ref{regular-qsymm} and Proposition \ref{KequalK},
$K^\U_i(A)=K_i(\cM(A, \U))$. Now $A=A_0 + A_1 + \dots$ is
$\Z_+$-graded with $A_0=\Bbbk$. Thus we may apply Theorem
\ref{graded} to compute its $\U$-equivariant K-groups. We have
$K^\U_i(A)=K^\U_i(\Bbbk)=K_i(\cP(\Bbbk, \U))$ for all $i=0, 1,
\dots$.

Note that $\cM(\Bbbk, \U)$ is the category of finite dimensional
left $\U$-modules of type-$(1, 1, \dots)$. As is well-known,
$\cM(\Bbbk, \U)$ is semi-simple, thus $\cP(\Bbbk, \U)=\cM(\Bbbk,
\U)=\Umod$.
\hfill$\square$

In particular $K^\U_0(A)$ is the Grothendieck group of $\Umod$.

\subsection{Examples} 
In this section, we consider examples of quantum symmetric algebras
arising from natural modules for the quantum groups associated with
the classical series of Lie algebras. These quantum symmetric algebras
also feature prominently in the study of the invariant theory of 
quantum groups \cite{LZZ}.

\begin{example}{\em Coordinate algebra of a quantum matrix.}
A familiar example of quantum symmetric algebras is $\cO(M(m, n))$,
the coordinate algebra of a quantum $m\times n$ matrix. It is
generated by $x_{i j}$ ($1\le i\le m$, $1\le j\le n$) subject to the
following relations
\begin{eqnarray}\label{qmatrix}
\begin{aligned}
x_{ij}x_{ik}&=
           q^{-1}x_{ik}x_{ij}, && j<k,\\
x_{ij}x_{kj}&=
            q^{-1}x_{kj}x_{ij}, && i<k,\\
x_{ij}x_{kl}&=x_{kl}x_{ij}, && i<k,
j>l,\\
x_{ij}x_{kl}&=x_{kl}x_{ij} - (q-q^{-1})x_{il}x_{kj},  && i<k, j<l.
\end{aligned}
\end{eqnarray}

It is well known that this is a module algebra over $\U_q(sl_n)$
with a PBW basis consisting of ordered monomials of the elements
$x_{i j}$. The $\U_q(sl_n)$-action on $\cO(M(m, n))$ can be
described as follows. For each $i$, the subspace $\oplus_{j=1}^n
\Bbbk x_{i j}$ is isomorphic to the natural module for $\U_q(sl_n)$.
Thus $\cO(M(m, n))$ is a quadratic algebra of the
$\U_q(sl_n)$-module $V$ which is the direct sum of $m$ copies of the
natural module.

In particular, when $m=1$, all relations but the first of
\eqref{qmatrix} are vacuous, and we obtain the quantised coordinate
algebra of affine $n$-space.

By \cite[\S I]{BK}, $\cO(M(m, n))$ is left Noetherian for all $m$
and $n$. By Theorem \ref{usual}, the ordinary algebraic K-groups of
$\cO(M(m, n))$ are given by $K_i(\cO(M(m, n)))=K_i(\Bbbk)$ for all
$i$. Theorem \ref{q-symmalg} also applies, and we have
\[
K^{\U_q(sl_n)}_i(\cO(M(m, n))) \cong K_i(\text{$\U_q(sl_n)$-{\bf mod}}),
\quad \text{for all $i$}.
\]
\end{example}

\begin{example}{\em Quantum symmetric algebras associated with the natural modules for
$\U_q(so_m)$ and $\U_q(sp_{2n})$.} We first briefly recall the
construction given in \cite{LZZ}.

An important structural property of the quantum group $\U=\U_q(\fg)$
associated with a simple Lie algebra $\fg$ is the braiding 
of its module category provided by a
universal $R$-matrix \cite{D}. We may think of this as an
invertible element in some completion of $\U\otimes_\Bbbk\U$, which
satisfies the following relations
\begin{eqnarray}
\begin{aligned}
&R \Delta(x) = \Delta'(x) R, \quad \forall x\in \U, \label{R1}&\\
&(\Delta\otimes\id)R = R_{1 3} R_{2 3}, \quad (\id\otimes\Delta)R=
R_{1 3} R_{1 2}, & \label{R2} \\
&(\epsilon\otimes \id)R = (\id\otimes\epsilon)R=1\otimes 1,&
\end{aligned}
\end{eqnarray}
where $\epsilon$ is the co-unit and $\Delta'$ is the opposite
co-multiplication. Here the subscripts of $R_{1 3}$ etc. have the
usual meaning as in \cite{D}. It follows from the second line of
\eqref{R2} that $R$ satisfies the celebrated Yang-Baxter equation
$
R_{1 2} R_{1 3} R_{2 3} = R_{2 3} R_{1 3} R_{1 2}.
$

Given a finite dimensional $\U$-module $V$, let $R_{V, V}$ denote
the automorphism of $V\otimes V$ defined by the universal
$R$-matrix of $\U$. Let $P: V\otimes V\longrightarrow V\otimes V$,
$v\otimes w \mapsto w\otimes v$, be the permutation, and define
$\check{R} =P R_{V, V}$. Then $\check{R}\in \End_{\U}(V\otimes V)$
by \eqref{R2}, and $\check R$
has characteristic polynomial of the form
\[ \prod_{i=1}^{k_+} \left(x - q^{\chi_i^{(+)}}\right)
\prod_{j=1}^{k_-} \left(x + q^{\chi_i^{(-)}}\right),
\]
where $\chi^{(+)}_i$ and $\chi^{(-)}_i$ are integers, and $k_{\pm}$
and positive integers. Consider the $\U$-submodule $I_-$ of $V\otimes V$
defined by
\begin{eqnarray}\label{relations}
I_-= \prod_{i=1}^{k_+} \left(\check{R} -
q^{\chi_i^{(+)}}\right)(V\otimes V).
\end{eqnarray}

We may then form the $\U$-module algebra $\Bbbk\{V, I_-\}$
associated to $V$ and $I_-$.
There is a classification in \cite{BZ, Zw} of those irreducible
$\U$-modules $V$ satisfying the condition
that the corresponding quadratic algebras
$\Bbbk\{V, I_-\}$ admit PBW bases. In particular, the natural
modules of the quantum groups associated with the classical Lie
algebras all have this property \cite{Zw}.

Recall from \cite{LZZ} that if  $A$ and $B$ are locally finite $\U$-module
algebras, $A\otimes_{\Bbbk}B$ becomes a $\U$-module algebra
if its multiplication is
twisted by the universal $R$-matrix. Explicitly, if we write
$R=\sum_t \alpha_t\otimes \beta_t$, then for all $a, a'\in A$ and
$b, b'\in B$,
\[
(a\otimes b)(a'\otimes b') = \sum_{t} a (\beta_t\cdot a')\otimes
(\alpha_t\cdot b) b'.
\]

If $C$ is a third locally finite $\U$-module algebra, the
$\U$-module algebras $(A\otimes B)\otimes C$ and $A\otimes (B\otimes
C)$ are canonically isomorphic \cite{LZZ}.
Therefore, given $\Bbbk\{V, I_-\}$ associated with an irreducible
finite dimensional $\U$-module, we have locally finite
$\U$-module algebras $\Bbbk\{V, I_-\}^{\otimes m}$ for each positive
integer $m$.

For any vector space $W$ we use the notation $W^n=\oplus^n W$.

\begin{theorem}\label{mSq}
Let $V$ be the natural module of $\U_q(so_m)$ or $\U_q(sp_{2n})$,
and let $I_-$ be the $\U$-submodule of $V\otimes V$ defined by
\eqref{relations}. Then $S_q(V^m): =\Bbbk\{V, I_-\}^{\otimes m}$ is
a Noetherian quantum symmetric algebra for every $m$.
\end{theorem}
Here by Noetherian we mean that the algebra is both left
and right Noetherian.

To prove the theorem, we require some preliminaries.
The following result from \cite{BK} will be of crucial importance.
\begin{lemma}\cite[Proposition I.8.17]{BK}\label{BK-lemma}
Let $A$ be an associative algebra over $\Bbbk$. Let $u_1, u_2,
\dots, u_N$ be a finite sequence of elements which generate $A$.
Assume that there exist scalars $q_{i j}\in\Bbbk^\times$, $\alpha_{i
j}^{s t}, \beta_{i j}^{s t}\in \Bbbk$ such that for all $i<j$,
\begin{eqnarray}\label{Noether-relation}
u_j u_i = q_{i j} u_i u_j
+ \sum_{s=1}^{i-1}\sum_{t=1}^N \left(\alpha_{i j}^{s t} u_s u_t +
\beta_{i j}^{s t} u_t u_s\right),
\end{eqnarray}
then $A$ is Noetherian.
\end{lemma}
Observe in particular that the algebra $\tilde A$ presented in terms of
the generators $u_1, \dots, u_N$ and the relations
\eqref{Noether-relation} is Noetherian. Thus any algebra defined
by the same generators subject to \eqref{Noether-relation} and extra relations
is a quotient of $\tilde A$, and hence is also Noetherian.

Let $V$ be the natural $\U_q(so_{2n})$-module, and 
let $\{v_a\mid a=1,\dots,2n\}$ be a basis of weight vectors of $V$,
with weights decreasing as $a$ increases.
Order this basis in the natural way: $v_1, v_2, \dots, v_{2n}$.
Then $\Bbbk\{V, I_-\}$ is generated by $v_a$ ($1\le a\le 2n$)
with the following relations
\begin{eqnarray}\label{vv-relations}
\begin{aligned}
& v_b v_a = qv_a  v_b , \quad a<b\le 2n, \ a+b\ne 2n+1,  \\
& v_{n+1} v_n = v_n v_{n+1},\\
&v_{2n-i} v_{i+1}= q^2 v_{i+1} v_{2n-i} -q v_i  v_{2n+1-i} +  q v_{2n+1-i} v_i, \quad  i\le n-1.
\end{aligned}
\end{eqnarray}

If we ignore the relation arising from the $i=n-1$
case of the third equation in \eqref{vv-relations},
the remaining relations define an algebra which satisfies the conditions
of Lemma \ref{BK-lemma}.  Therefore, $\Bbbk\{V, I\}$ is Noetherian.
It is known \cite{BZ, Zw} that $\Bbbk\{V, I\}$ admits a PBW basis, and hence
is a Noetherian quantum symmetric algebra.

We now turn to Theorem \ref{mSq}.

\begin{proof}[Proof of Theorem \ref{mSq}]
It was shown in \cite{LZZ} that $S_q(V^m)$ admits a
PBW basis for every $m$, whence
it suffices to prove that this algebra is Noetherian.

If $V$ is the natural $\U_q(so_{2n})$-module,
we may regard $S_q(V^m)$ as generated by $X_{i a}$
with $1\le i\le  m$ and $1\le a\le 2n$,
subject to the relations R(1) and R(2) below:
\begin{enumerate}
\item[R(1):] For any $i$,
\[
\begin{aligned}
& X_{i b} X_{i a} = q X_{i a}  X_{i b} ,
\quad a<b\le 2n, \ a+b\ne 2n+1,  \\
& X_{i, n+1} X_{i n} = X_{i n} X_{i, n+1},\\
&X_{i, 2n-s} X_{i, s+1}= q^2 X_{i, s+1} X_{i, 2n-s}
-q X_{i s}  X_{i, 2n+1-s} +  q X_{i, 2n+1-s} X_{i s}, \quad  s\le n-1.
\end{aligned}
\]
\item[R(2):] For $i<j$ and all $a, b$,
\[
X_{j b} X_{i a} =  q_{a b}^{-1}  X_{i a} X_{j b}
+ \sum_t (\beta'_t\cdot X_{i a} ) (\alpha'_t\cdot X_{j b} ),
\]
where $q_{a b}$ is $q^{-1}$ if $a=b$,
is $q$ if $a+b=2n+1$,  and is $1$ otherwise.
\end{enumerate}
Here $\sum_a \Bbbk X_{i a} \cong V$
as $\U$-module for each $i$, and 
we use the form $R=K+\sum_t \alpha'_t\otimes \beta'_t$
for the universal $R$-matrix, where $K$ acts by $K (X_{j b}\otimes X_{i a})
= q_{a b}^{-1}X_{j b}\otimes X_{i a}$.
Actions of $\alpha'_t$ (resp. $\beta'_t$) increase (resp. decrease) weights. We have
$\beta'_t\cdot X_{i a}= \zeta_{a t} X_{i a_t}$ and
$\alpha'_t\cdot X_{j b}= \eta_{t b} X_{j b_t}$ for some
$a_t>a$ and $b_t<b$, where $\zeta_{t a}$ and $\eta_{t b}$
are scalars such that $\zeta_{t a}\eta_{t b}\ne 0$ only for finitely many $t$.

Order the elements $X_{i a}$ as follows:
\[
X_{m 1}, X_{m 2}, \dots, X_{m, 2n}; X_{m-1, 1}, X_{m-1, 2}, \dots, X_{m-1, 2n};\dots;
X_{1 1}, X_{1 2}, \dots, X_{1, 2n}.
\]
Note that relations R(1) are the same as \eqref{vv-relations}
with $v_a$ replaced by $X_{i a}$, and the order of the elements
agrees with that of the $v_a$. The relations R(2) may be re-written as
\[
\begin{aligned}
X_{i a} X_{j b}&= q_{a b} X_{j b} X_{i a}
- q_{a b}\sum_t (\beta'_t\cdot X_{i a} ) (\alpha'_t\cdot X_{j b} )\\
&= q_{a b} X_{j b} X_{i a} - q_{a b}\sum_t \zeta_{a t} \eta_{t b} X_{i a_t}  X_{j b_t}
\end{aligned}
\]
for $i<j$ and all $a, b$. These relations are in the form of \eqref{Noether-relation}.
Therefore $S_q(V^m)$ meets the conditions of
Lemma \ref{BK-lemma}, and hence is Noetherian. This completes the proof
for the case of $\U_q(so_{2n})$

When $V$ is the natural $\U_q(so_{2n+1})$-module or natural
$\U_q(sp_{2n})$-module, there are defining relations of $S_q(V^m)$
analogous to those in the $\U_q(so_{2n})$ case \cite{LZZ}, and
similar reasoning shows that $S_q(V^m)$ is also Noetherian.
We leave the details of the proof in these cases to the reader.
\end{proof}

In view of Theorem \ref{mSq}, we may now apply Theorem \ref{usual}
and Theorem \ref{q-symmalg} to compute the ordinary
and equivariant K-groups of $S_q(V^m)$. We have
\[
K_i(S_q(V^m)) \cong K_i(\Bbbk), \quad
K^{\Uq}_i(S_q(V^m)) \cong K_i(\text{$\Uq$-{\bf mod}}), \quad
\text{for all $i$},
\]
where $\Uq$ is $\U_q(so_m)$ or $\U_q(sp_{2n})$.
\end{example}

\begin{example}
We consider the Weyl algebra $W_q$ of degree $1$ over
$\Bbbk=\C(q)$. It is  generated by $x, y$ subject to the relation
\[ x y - q^{-1} y x =1. \]
This is a module algebra over $\U_q(sl_2)$ if we
identify $\{x, y\}$ with the standard basis of the natural $\U_q(sl_2)$-module
$\Bbbk^2$.

Let $F_i W_q$ be the span of the elements $x^{j-t} y^t$ with $i\ge j
\ge t\ge 0$. Then we have a complete ascending filtration $0\subset
F_0 W_q \subset F_1 W_q \subset F_2 W_q\subset \dots$, which is
stable under the $\U_q(sl_2)$-action. The associated graded algebra
$gr(W_q)$ is the $\U_q(sl_2)$-module algebra generated by $x, y$
subject to the relation $x y = q^{-1} y x$. By Lemma
\ref{regular-qsymm}, this algebra is left regular.

Therefore Theorem \ref{usual} and Theorem \ref{q-symmalg} apply 
to $W_q$, we obtain 
$K_i(W_q) = K_i(\Bbbk)$ and
$K^{\U_q(sl_2)}_i(W_q)=K_i(\cP(\Bbbk, \U_q(sl_2)))$ ($i\ge 0$)
for the Quillen K-groups and equivariant K-groups respectively.
In the present case, $\cM(\Bbbk, \U_q(sl_2))=\cP(\Bbbk,
\U_q(sl_2))=\U_q(sl_2)\text{-\bf mod}$ is the category of finite
dimensional $\U_q(sl_2)$-modules of type-$(1, \dots, 1)$. 
Hence $K^{\U_q(sl_2)}_i(W_q)=K_i(\U_q(sl_2)\text{-\bf mod})$ for all $i\ge 0$. 
\end{example}

\section{Quantum homogeneous spaces}\label{q-homo-spaces}

Quantum homogeneous spaces \cite{GZ} are a class of noncommutative
geometries with quantum group symmetries, which have been widely studied
(see, e.g., \cite{DLPS, Maj} and the references therein). We shall develop their
equivariant K-theory in this section. We mention that the
equivariant $K_0$-groups of quantum homogeneous spaces have been
determined in \cite{ZZ}.

\subsection{Quantum homogenous spaces}\label{spaces}
We recall from \cite{GZ, ZZ} some background material on  quantum homogeneous spaces,
which will be needed later.
Let $V$ be an object in $\U$-{\bf mod}, and let $\pi: \U\longrightarrow \End_\Bbbk(V)$
be the corresponding matrix representation of $\U$ relative to some basis of $V$. Then there exist
elements $t_{i j}$ ($i, j=1, 2, \dots, \dim V$) in the dual $\U^*$ of $\U$ such that for
any $x\in \U$, we have $\pi(x)_{i j} = \langle t_{i j}, x\rangle$ for all $i, j$.
The $t_{i j}$ will be referred to as 
the coordinate functions of the finite dimensional representation $\pi$.
It follows from standard facts in Hopf algebra theory \cite{M} that
the coordinate functions of all the various $\U$-representations associated with
the $\U$-modules in $\U$-{\bf mod} span a Hopf algebra $\cO_q(\U)$,
which is a Hopf subalgebra of the finite dual of $\U$ (see \cite{M} for this notion).
We shall denote the co-multiplication
and the antipode of $\cO_q(\U)$ by $\Delta$ and $S$ respectively,
but the co-unit will be denoted by $\epsilon_0$. Note that the
co-unit (resp. unit) of $\U$ becomes the unit (resp. co-unit) of
$\cO_q(U)$.

We shall denote $\cO_q(\U_q(\fg))$ by $A_\fg$ for brevity.
There exist two natural actions $R$ and $L$ of $\U$ on $A_\fg$
\cite{GZ}, which correspond to left and right translation in
the context of Lie groups. These actions are respectively defined by
\begin{eqnarray*}
R_x f=\sum_{(f)}f_{(1)}<f_{(2)},x>, && L_x
f=\sum_{(f)}<f_{(1)},S(x)>f_{(2)}
\end{eqnarray*}
for all $x\in{\U}$ and $f\in{A_\fg}$, where
as $f(x)=\langle f , x \rangle$ for any $x\in\U$ and $f$ in the finite
dual of $\U$. These two actions clearly commute by the coassociativity
of the Hopf structure on $A_\fg$.

It can be shown that $A_\fg$ forms a $\U$-module algebra under both
actions. However, some care needs to be exercised in the case of
$L$, as for any $f, g\in A_\fg$ and $x\in\U$, we have
\[
\begin{aligned}
L_x(f g) &= \sum_{(f), (g)} \langle f_{(1)}g_{(1)}, S(x)\rangle
f_{(2)}g_{(2)} \\ & = \sum_{(f), (g), (x)} \langle f_{(1)},
S(x_{(2)})\rangle \langle g_{(1)}, S(x_{(1)})\rangle
f_{(2)}g_{(2)}\\
&= \sum_{(x)} L_{x_{(2)}}(f) L_{x_{(1)}}(g).
\end{aligned}
\]
This shows that  under the action $L$,  $A_\fg$ forms a module
algebra over $\U'$, which is $\U$, with the opposite co-multiplication
$\Delta'$.

Let $\Theta$ be a subset of $\{1, 2, \dots, r\}$, where $r$ is the
rank of $\fg$. We denote by $\Ul$ the Hopf subalgebra of $\U$
generated by the elements of $\{k_i^{\pm} \mid 1\le i\le r\}\cup \{
e_j,f_j \mid j\in{\Theta}\}$. We denote by $\Ulmod$ the category of
finite dimensional left $\Ul$-modules of type-$(1,
\dots, 1)$. This category is semisimple.

Following \cite{GZ}, we define
\begin{eqnarray}\label{Agl}
A= \left\{ f\in A_\fg\, \mid \, L_x (f) = \epsilon(x) f, \ \forall
x\in\Ul \right\}=A_\fg^{\Ul}.
\end{eqnarray}
This is the submodule of $\Ul$-invariants of $A_\fg$.

The following result is fairly straightforward (cf. \cite{GZ, ZZ}).
\begin{theorem} \label{noetherian}
The subspace $A$ forms a locally finite $\U$-module algebra under the action $R$.
Furthermore, $A$ is (both left and right) Noetherian.
\end{theorem}
Indeed, since $\Ul$ is a Hopf subalgebra of $\U$, it follows from
the definition that $A$ is a subalgebra of $A_\fg$. Since left
and right translations commute, the $\U$-module algebra structure of
$A_\fg$ under $R$ descends to $A$. Being a subalgebra of $A_\fg$
which is contained in the finite dual
of $\U$, $A$ must be locally finite under the $\U$-action $R$.
The fact that $A$ is Noetherian
is proved in \cite{ZZ}.

It was shown in \cite{GZ} that the algebra $A$ is the
natural quantum analogue of the algebra of complex valued (smooth) functions on
the real manifold $G/K$ for a compact connected Lie group $G$
and a closed subgroup $K$, where the
Lie algebras ${\mathcal Lie}(G)$ and ${\mathcal Lie}(K)$ have
complexifications $\fg$ and $\fl$ respectively.
Thus the noncommutative space determined by the algebra $A$
is referred to as a {\em quantum homogeneous space} in \cite{GZ}.

\begin{remark}
A quantum homogeneous space defined this way is a
quantisation of the real manifold underlying a
compact homogeneous space.
A complex structure (in a generalised  sense) on the quantum homogeneous space
was discussed and used in establishing Borel-Weil type theorems in \cite{GZ}.
\end{remark}

\subsection{Equivariant $K$-theory of quantum homogeneous
spaces}\label{K-theory}

For any object $\Xi$ of $\Ulmod$, we define
the induced $\U$-module as follows:
\begin{equation}\label{sections}
\begin{aligned}
\cS(\Xi):=&\left\{\zeta\in  \Xi\otimes A_\fg    \left|
\sum_{(x)}(x_{(1)} \otimes L_{x_{(2)}}) \zeta = \epsilon(x)\zeta, \
\forall x\in\Ul\right. \right\}\\
=&\left(\Xi\otimes_\Bbbk A_\fg\right)^{\Ul},
\end{aligned}
\end{equation}
where $\Ul$ acts on the tensor product via $(\id\otimes L)\circ \Delta$.

Then $\cS(\Xi)$ is both a left $A$-module and left $\U$-module with
$A$ and $\U$ actions defined, for $b\in A$, $x\in\U$ and
$\zeta=\sum v_i \otimes a_i\in\cS(\Xi)$, by
\begin{eqnarray}
b\zeta &=& \sum v_i\otimes b a_i, \label{A-action}\\ x \zeta
&=&(\id_\Xi \otimes R_x)\zeta = \sum v_i \otimes
R_x(a_i).\label{U-action}
\end{eqnarray}
We have
\begin{eqnarray*}
x(b\zeta) &= \sum v_i\otimes R_x(b a_i)
&=\sum_{(x)}R_{x_{(1)}}(b)(x_{(2)}\zeta), \quad \text{for $b\in A$}.
\end{eqnarray*}
Thus $\cS(\Xi)$ indeed forms a $\U$-equivariant ${A}$-module.

The following results were proved in \cite{GZ, ZZ}.
\begin{theorem}\label{key}
\begin{enumerate}
\item Let $V$ be the restriction of a finite dimensional
$\U$-module to a $\Ul$-module. Then $\cS(V)\cong V\otimes_{\Bbbk}A$
in $\cM(A, \U)$.
\item For any object $\Xi$ in $\Ulmod$, $\cS(\Xi)$ is
an object of $\cP(A, \U)$.
\end{enumerate}
\end{theorem}

Recall from \cite{GZ} that part (2) of the theorem follows from part
(1). Indeed, $\Xi$ can always be embedded in the restriction of some
finite dimensional $\U$-module as a direct summand. That is, there
exist a finite dimensional $\U$-module $W$ and a $\Ul$-module
$\Xi^\bot$ such that $W\cong\Xi\oplus\Xi^\bot$ as $\Ul$-module. It
follows from part (1) of Theorem \ref{key} that
$\cS(\Xi)\oplus\cS(\Xi^\bot)\cong W\otimes A$. By Corollary
\ref{projective}, $\cS(\Xi)$ is in $\cP(A, \U)$.

We extend \eqref{sections} to a covariant functor
\begin{eqnarray}\label{S-functor}
\cS: \Ulmod\longrightarrow \cP(A, \U),
\end{eqnarray}
which acts on objects of $\Ulmod$ according to \eqref{sections}
and sends a morphism $f$ to $f\otimes{\rm{id}}_{A_\fg}$. Since
$\Ulmod$ is semi-simple and $\cS(V\oplus W) = \cS(V)\oplus \cS(W)$
for any direct sum $V\oplus W$ of objects in $\Ulmod$, the functor
$\cS$ is exact.

Let $I=\{f \in A |f(1)=0 \}$; this is a maximal two-sided ideal of
$A$. We have $A/I\cong\Bbbk$. For any $x\in\Ul$ and $a\in I$,
$\langle R_x(a), 1\rangle = \epsilon(x) a(1)=0$, thus $I$ forms a
$\Ul$-algebra under the restriction of the action $R$. This
implies that for any $\U$-equivariant ${A}$-module $M$,
$I M$ is a $\Ul$-equviaraint ${A}$-submodule of $M$. This can be
seen from the following calculation: for any $a\in I$ and $m\in M$,
we have $x(a m) = \sum_{(x)} R_{x_{(1)}}(a) x_{(2)}m\in I M$ for all
$x\in\Ul$.

Given a $\U$-equivariant $A$-module $M$, let $ M_0=M/IM $;
this is a $\Ul$-equivariant $A$-module in which $a\in A$
acts as $a(1)\in\Bbbk$. Denote the natural surjection by
\begin{eqnarray}\label{p-map}
p:  M\longrightarrow M_0.
\end{eqnarray}
This is an $A$-$\Ul$-linear map.

For any object $M$ in $\cM(A, \U)$, we can find a finite
dimensional $\U$-submodule $W$ which generates $M$. By Lemma
\ref{free-module}, the $A$-$\U$-map $A\otimes_{\Bbbk}
W\longrightarrow M$, $a\otimes w\mapsto a w$, is surjective.
Therefore $M_0 = p(W)$. Note that $W$ is semi-simple as $\U$-module and
hence also as $\Ul$-module. Thus $M_0$ belongs to $\Ulmod$.

We therefore have a covariant functor
\[
\cE: \cM(A, \U)\longrightarrow \Ulmod,
\]
which sends an object $M$ in
$\cM(A, \U)$ to $M_0$, and is defined on morphisms in the obvious
way. We may restrict this functor to the full subcategory $\cP(A,
\U)$ to obtain a covariant functor
\begin{eqnarray}\label{evalute}
\cE_\cP: \cP(A, \U)\longrightarrow \Ulmod.
\end{eqnarray}

\begin{theorem} \label{main}
The functors $\cS:\Ulmod\longrightarrow \cP(A, \U)$ and $\cE_\cP:
\cP(A, \U)\longrightarrow\Ulmod$ respectively defined by
\eqref{S-functor} and \eqref{evalute} are mutually inverse equivalences of
categories.
\end{theorem}

We will prove the theorem in Section \ref{proof-main}. It
has the following consequence.

\begin{corollary}\label{main-1}
There is an isomorphism of abelian groups
\[ K^\U_i(A)  \cong K_i(\Ulmod)\]
for  each   $i\ge 0$, where $\Ulmod$ is the category of finite
dimensional left $\Ul$-modules of type-$(1, \dots, 1)$ (which is semi-simple).
\end{corollary}

This implies in particular that $K^\U_0(A)$ is isomorphic to the
Grothendieck group of $\Ulmod$, a result proved in
\cite{GZ}.

\subsection{Proof of Theorem \ref{main}}\label{proof-main}

We now prove Theorem \ref{main} by means of a
series of lemmas.

For any $V$ in $\Ulmod$, $\cS(V)$ is the subspace of
$\Ul$-invariants in $V\otimes_{\Bbbk} A_\fg$ with respect to the
action $\id_V\otimes L_{\Ul}$. The linear map 
$V \otimes_{\Bbbk} A_\fg\longrightarrow V$ given by
\[
 \quad \zeta=\sum_{i}
 v_i\otimes f_i \quad \mapsto  \quad  \zeta(1) = \sum_{i} f_i(1) v_i,
\]
induces a linear map (`evaluation')
\begin{eqnarray}\label{ev}
ev: \cS(V)\longrightarrow V, \quad \zeta \mapsto\zeta(1).
\end{eqnarray}
Note that $\cS(V)$ is an $A$-$\Ul$-module with the standard actions
defined by \eqref{A-action} and the restriction of \eqref{U-action}.
We may also define an $A$-module structure on $V$ in which each $a\in A$
acts as scalar multiplication by $a(1)$. This makes $V$ into an
$A$-$\Ul$-module.

\begin{lemma}
Any $V$ in $\Ulmod$ may be regarded as an $A$-$\Ul$-module as above.
The map \eqref{ev} is $A$-$\Ul$-linear.
\end{lemma}
\begin{proof}
Given its importance, we give a brief proof of this lemma.
For any $\zeta\in \cS(V)$ and $u\in\Ul$, we have
\[
\sum_{(u)} (u_{(1)}\otimes L_{u_{(2)}})\zeta\otimes u_{(3)} =
\zeta\otimes u
\]
by the $\Ul$-invariance of $\cS(V)$. Using $\sum_{(u)}
(u_{(1)}\otimes L_{u_{(2)}})\zeta(u_{(3)})=u\cdot ev(\zeta)$ and
$\zeta(u) = ev((\id_V\otimes R_u)\zeta)$, we obtain
$u\cdot ev(\zeta)= ev((\id_V\otimes R_u)\zeta)$.  Finally, for any $a\in
A$, we have $ev(a\zeta) = a(1) ev(\zeta)$. This completes the proof.
\end{proof}

In view of the lemma and the fact that the map \eqref{p-map} is
$\Ul$-linear, we have the following $\Ul$-map for each $V$:
\begin{eqnarray}\label{epsilon}
\epsilon_V: \cE\circ\cS(V) \longrightarrow V, \quad p(\zeta) \mapsto
ev(\zeta)=\zeta(1).
\end{eqnarray}

\begin{proposition}\label{epsilon-is}
The map $\epsilon_V$ defined by \eqref{epsilon} for each object $V$
in $\Ulmod$ gives rise to a natural transformation
\begin{eqnarray}
\epsilon: \cE\circ\cS\longrightarrow \id_{\Ulmod},
\end{eqnarray}
which is in fact a natural isomorphism.
\end{proposition}
\begin{proof}
For any map $\alpha: V\longrightarrow V'$ in $\Ulmod$,
$\cE\circ\cS(\alpha)$ is given by
\[
\cE\circ\cS(\alpha)(p(\zeta))=p((\alpha\otimes \id_{A})\zeta), \quad
\forall p(\zeta)\in\cE\circ\cS(V).
\]
Now $\epsilon_V\circ p((\alpha\otimes \id_{A})\zeta) =
((\alpha\otimes \id_{A})\zeta)(1) = \alpha(\zeta(1))$. This proves
the commutativity of the following diagram
\[
\begin{array}{c c c c c}
&\cE\circ\cS(V)& \stackrel{\epsilon_V}{\longrightarrow}& V& \\
&\downarrow&                        &\downarrow \alpha&   \\
&\cE\circ\cS(V')& \stackrel{\epsilon_{V'}}{\longrightarrow}& V'&
 \end{array},
\]
where the left vertical map is $\cE\circ\cS(\alpha)$.  Hence
$\epsilon$ is a natural transformation between the functors
$\cE\circ\cS$ and $\id_{\Ulmod}$.

We have already noted that any module $V$ in $\Ulmod$ may be embedded in
the restriction of a finite dimensional $\U$-module $W$ as a direct
summand, that is, $W\cong V\oplus V^\bot$ for some $\Ul$-module
$V^\bot$. Now $\cS(W) \cong W\otimes_{\Bbbk}A$ as $A-\U$-module,
by Theorem \ref{key}(2). Using this isomorphism, we obtain that
$\cE\circ\cS(W)\cong W$, $ev(\cS(W))=W$, and that $\epsilon_W:
\cE\circ\cS(W) \longrightarrow W$ is an isomorphism. Since
$\cS(V)\oplus\cS(V^\bot) \cong W\otimes_{\Bbbk} A$, we have
$\epsilon_W = \epsilon_V\oplus\epsilon_{V^\bot}$. As
$\epsilon_W$ is an isomorphism, both $\epsilon_V$ and
$\epsilon_{V^\bot}$ are isomorphisms.
\end{proof}

It follows from Appendix \ref{coalgebra} that there is a one to one
correspondence between locally finite left $\U$-modules (of
type-$(1, 1, \dots, 1)$) and right $A_\fg$-comodules. For any
locally finite left $\U$-module $M$, denote the corresponding right
$A_\fg$-comodule structure map by
\[
\delta_M: M \longrightarrow M\otimes_{\Bbbk} A_\fg, \quad w\mapsto
\delta_M(w) = \sum_{(w)} w_{(1)}\otimes w_{(2)}.
\]
Let $M$ be and object of $\cM(A, \U)$, and consider the
composition $ M \stackrel{\delta_M}{\longrightarrow}
M\otimes_{\Bbbk}A_\fg \stackrel{p\otimes\id}{\longrightarrow}
p(M)\otimes_{\Bbbk}A_\fg. $ Note that for any $m\in M$ and
$u\in\Ul$, we have
\[
\begin{aligned}
\sum p(u_{(1)} m_{(1)})\otimes L_{u_{(2)}}(m_{(2)}) &= \sum
p(m_{(1)})\otimes \langle m_{(2)}, u_{(1)}\rangle \langle m_{(3)},
S(u_{(2)})\rangle m_{(4)}\\& =\epsilon(u) \sum p(m_{(1)})\otimes
m_{(2)}.
\end{aligned}
\]
Hence the image of this map is contained in
$\cS(\cE(M))$, and we obtain a map
\[
\hat\delta_M=(p\otimes\id)\circ\delta_M: M \longrightarrow
\cS(\cE(M)),
\]
which is clearly $A$-linear. Moreover for any $m\in M$ and $x\in\U$, we
have
\[
\hat\delta_M(x m) = \sum p(m_{(1)})\otimes R_{x}m_{(2)} =x
\hat\delta_M(m).
\]
Thus the map is also $\U$-linear, and is therefore a homomorphism
of $A-\U$-modules. This leads to the following result, which was proved
in \cite[Proposition 3.8]{ZZ} in a slightly different form,
and was used there to show that the Grothedieck groups of $\cP(A, \U)$
and $\Ulmod$ are isomorphic.

\begin{lemma} \label{delta-2}
If $M$ is in $\cP(A, \U)$, then $\hat\delta_M: M\longrightarrow
\cS\circ\cE(M)$ is an isomorphism of $A-\U$-modules.
\end{lemma}

\begin{proposition}\label{delta-1}
The maps $\hat\delta_M$, for $M$ in $\cM(A, \U)$ define
a natural transformation $\id_{\cM(A, \U)}\longrightarrow
\cS\circ \cE$.
\end{proposition}
\begin{proof}
Let $\beta: M\longrightarrow N$ be a morphism in $\cM(A, \U)$. Then
\[ \cS\circ\cE(\beta)\hat\delta_M=
(p_N\otimes\id)(\beta\otimes\id)\delta_M, \] where $p_N$ is the map
\eqref{p-map} for $N$. Using the fact that
$(\beta\otimes\id)\delta_M = \delta_N\beta$, we obtain
$\cS\circ\cE(\beta)\hat\delta_M= \hat\delta_N\beta$. That is, the
following diagram commutes
\[
\begin{array}{c c c c c}
&M& \stackrel{\hat\delta_M}{\longrightarrow}&\cS\circ\cE(M)& \\
&\beta\downarrow&                        &\downarrow&   \\
&N& \stackrel{\hat\delta_N}{\longrightarrow}&\cS\circ\cE(N)&
 \end{array},
\]
where the right vertical map is $\cS\circ\cE(\beta)$.
\end{proof}

The next statement is an immediate consequence of
Lemma \ref{delta-1} and Lemma \ref{delta-2}.

\begin{proposition}\label{delta-is}
There is a natural isomorphism  $\hat\delta: \id_{\cP(A, \U)}
\longrightarrow \cS\circ \cE_\cP$.
\end{proposition}

\begin{proof}[Proof of Theorem \ref{main}]
It follows from Proposition \ref{epsilon-is} and
Proposition \ref{delta-is} that the categories $\Ulmod$ and $\cP(A,
\U)$ are equivalent.
\end{proof}

\begin{appendix}
\section{Comodules and smash products}\label{coalgebra}\label{smash}

The notions of module algebras and equivariant modules can also be
formulated in terms of co-algebras and comodules, which we 
discuss briefly here. Some of this material is used in
Section \ref{proof-main}.
We shall also discuss the relationship between
locally finite equivariant modules over a $\U$-module algebra $A$,
and finitely generated modules over the smash product of $A$ and $\U$.

Recall from Section \ref{proof-main}
the coordinate functions of the $\U$-representations associated with
the $\U$-modules in $\U$-{\bf mod} span a Hopf algebra $\cO_q(\U)$.

A comodule over $\cO_q(U)$ is a vector space $M$ with a
$\Bbbk$-linear map $\delta: M\longrightarrow M\otimes_\Bbbk
\cO_q(U)$ satisfying the following conditions:
\[
(\delta\otimes\id_{\cO_q(U)})\circ \delta =(\id_M\otimes\Delta)
\delta, \quad  (\id_M\otimes\epsilon_0)\delta = \id_M.
\]
We use Sweedler's notation $\delta(v)=\sum_{(v)} v_{(1)}\otimes
v_{(2)}$ for the co-action on $v\in M$, so that $v_{(1)}\in M$,
while $v_{(2)}\in {\cO_q(U)}$.

A locally finite left $\U$-module $M$ of type-$(1, \dots, 1)$
is naturally a right comodule
over $\cO_q(U)$, and vice versa. The comodule structure map
$\delta: M\longrightarrow M\otimes_\Bbbk \cO_q(U)$ and the module
structure map $\phi: \U\otimes M \longrightarrow M$ are related to
each other by
\[
\delta(v)(x) = \phi(x\otimes v), \quad \text{for all \ } x\in \U, \ v\in M.
\]
In this definition, 
local $\U$-finiteness of $M$ is needed in order for $\delta(v)$ to
lie in $M\otimes \cO_q(\U)$ for all $v\in M$.

If a $\U$-module is not locally finite, it does not
correspond to any comodule over $\cO_q(\U)$. The local finiteness
condition is built into the definition of comodules because
$\cO_q(U)$ is the Hopf algebra spanned by the coordinate functions of
$\U$-representations corresponding to objects in $\U$-{\bf mod}.

Now a locally finite $\U$-module algebra $A$ of type-$(1, \dots, 1)$
is nothing but an
associative algebra which is a right comodule over $\cO_q(\U)$
satisfying the conditions
\[
\begin{aligned}
&\delta_A(1_A) = 1_A\otimes \epsilon, \\
&\delta_A(a b) = \sum_{(a), (b)} a_{(1)} b_{(1)}
\otimes a_{(2)} b_{(2)}, \quad \forall a, b\in A,
\end{aligned}
\]
where $\delta_A$ is the comodule structure map of $A$.

Let $A$ be a locally finite $\U$-module algebra of type-$(1, \dots, 1)$, 
and as above, denote the
co-action of $\cO_q(\U)$ on any  $a\in A$ by $a\mapsto \sum_{(a)}
a_{(1)}\otimes a_{(2)}$.

A locally $\U$-finite, type-$(1, \dots, 1)$ , $\U$-equivariant
$A$-module $M$ is a left $A$-module, which is also a right
$\cO_q(\U)$-comodule, such that the $A$-module structure and
$\cO_q(\U)$-comodule structure $\delta_M: M \longrightarrow M\otimes
\cO_q(\U)$ are compatible in the sense that for all $a\in A$ and
$v\in M$,
\[
\delta_M(a v) = \sum_{(a), (v)} a_{(1)} v_{(1)} \otimes a_{(2)} v_{(2)}.
\]

A notion closely related to equivariant modules is the {\em smash
product} \cite[Definition 4.1.3]{M}. Given a $\U$-module algebra $A$,
one may construct the smash product $R:=A\#\U$, which is an
associative algebra with underlying vector space
$A\otimes_\Bbbk\U$ and multiplication defined by
\[  (a\otimes u) (b\otimes v) = \sum_{(u)}a (u_{(1)}\cdot b) \otimes u_{(2)} v \]
for all $a, b\in A$ and $u, v\in \U$.

It is easy to show that a left $\U$-equivariant $A$-module is in
fact a left $R$-module. However, a
finitely generated $R$-module need not be locally $\U$-finite,
and so 
is not generally in $\cM(A, \U)$. Therefore, $\cP(A, \U)$ is
different from the category of finitely generated projective
$R$-modules.

A case in point is the following example.
Take $A=\C(q)$ be equipped with trivial $\U$-action (through the
co-unit). In this case, both $\cM(A, \U)$ and $\cP(A, \U)$ coincide with
$\U$-{\bf mod}.  On the other hand, the smash product $R=A\#\U$ is $\U$ itself.
The category of finitely generated
projective $\U$-modules is totally different from the category
$\U$-{\bf mod}.

This shows that the equivariant K-theory of the $\U$-module algebra
$A$ introduced in Section \ref{definitions-K} is quite different 
from the usual K-theory of the smash product 
$R:=A\#\U$, a fact which we have already pointed 
out in Remark \ref{not-smash}.
\end{appendix}

\bigskip

\noindent {\bf Acknowledgements}.
We thank Ngau Lam for discussions throughout the course of this work. We also thank
Alan Carey for helpful correspondence and Peter Donovan for suggestions on
improving the first draft of this paper.  This work is supported
by the Australian Research Council.

\end{document}